\documentclass[10pt]{article}
\usepackage{amssymb,amsmath,amsfonts,amsthm,enumerate,color}
\usepackage[toc,page,title,titletoc,header]{appendix}
\usepackage{graphicx}
\usepackage{indentfirst}
\usepackage{multicol}
\usepackage{booktabs}
\usepackage{cases}

\numberwithin{equation}{section}

\newtheorem{Theorem}{Theorem}[section]

\newtheorem{Remark}[Theorem]{Remark}
\numberwithin{equation}{section}

\def \Vh0{\stackrel{\circ}{V}_h} \def\to{\rightarrow}

\newcommand{\q}{\quad}

\def\ms{\medskip}  \def\ss{\smallskip}

\newcommand{\lc}
{\mathrel{\raise2pt\hbox{${\mathop<\limits_{\raise1pt\hbox
{\mbox{$\sim$}}}}$}}}

\newcommand{\gc}
{\mathrel{\raise2pt\hbox{${\mathop>\limits_{\raise1pt\hbox{\mbox{$\sim$}}}}$}}}

\newcommand{\ec}
{\mathrel{\raise2pt\hbox{${\mathop=\limits_{\raise1pt\hbox{\mbox{$\sim$}}}}$}}}

\def\bb{\begin{equation}} \def\ee{\end{equation}}

\def\beqn{\begin{eqnarray}}  \def\eqn{\end{eqnarray}}

\def\beqnx{\begin{eqnarray*}} \def\eqnx{\end{eqnarray*}}

\def\bn{\begin{enumerate}} \def\en{\end{enumerate}}

\def\bd{\begin{description}} \def\ed{\end{description}}

\makeatletter

\newenvironment{figurehere}
  {\def\@captype{figure}}
  {}
\makeatother

\begin{document}

\title{A Multilevel Sampling Method \\ for Detecting Sources in a Stratified Ocean Waveguide\footnote{The work was initiated when the first and second authors were visiting CUHK and 
supported by the Direct Grant for Research from CUHK.}}
\author{Keji Liu\thanks{Shanghai Key Laboratory of Financial Information Technology, Laboratory Center, Shanghai University of Finance and Economics, 777 Guoding Road, Shanghai 200433, P.R.China.
 ({\tt liu.keji@mail.shufe.edu.cn}). The work of this author was substantially supported by the Science and Technology Commission of Shanghai Municipality under grants 14511107202 and 15511107302.} 
~\q Yongzhi Xu\thanks{Department of Mathematics, University of
Louisville, Louisville, KY 40245, USA. 
({\tt
ysxu0001@louisville.edu})}
~\q Jun Zou\thanks{Department of Mathematics, The Chinese University
of Hong Kong, Shatin, Hong Kong.
({\tt zou@math.cuhk.edu.hk}). The work of this author was substantially supported by Hong Kong RGC grants (projects 14306814 and 405513).}
}
\date{}

\maketitle\textbf{Abstract.} In the reconstruction process of sound waves in a 3D stratified waveguide, 
a key technique is to effectively reduce the huge computational demand. 
In this work, we propose an efficient and simple multilevel reconstruction method to help locate the accurate position 
of a point source in a stratified ocean. The proposed method can be viewed as a direct sampling method since 
no solutions of optimizations or linear systems are involved. The novel method exhibits several strengths: fast convergence, robustness against noise, advantages in computational complexity and applicability for a very small number of receivers.

\ss
\textbf{Key words.} Stratified ocean, acoustic point source, direct sampling method.

{\bf MSC classifications.} 35R30, 41A27, 76Q05.

\maketitle

\section{Introduction}
The primal goal of  this work is to develop an effective numerical method for locating underwater sound sources 
in a stratified ocean environment.
Sound propagation in the ocean environment has great importance, and 
a large number of experiments have been carried out for long-range propagations. Many numerical
schemes have been proposed for range-dependent propagations and scattering
problems in the ocean environment \cite{JF, JKPS, PR}. Mathematical investigations and modeling of wave
propagations in underwater environment are widely available in literature; see 
\cite{AK, Bre, GL1, GL2, GX1, GX,  GXT, JJ, J, Wil} and the references therein.

Underwater sound propagation has been studied based on the physical principles of the acoustics. As a simple but reasonably realistic model for studying the effect of the underwater sound wave propagation we shall 
consider the stratified media \cite{AK, Bre} in this work.  To simplify the configuration but still retain the physical features of interest, we model the problem as a horizontally stratified waveguide, and assume 
the sound speed $c$ depends essentially on the depth of the waveguide. A typical sound velocity profile 
consists of the surface channel, the thermocline and isothermal layers; see
Figure \ref{fig:SVP} \cite{L}. 
As we know, the surface channel is mainly formed 
due to a shallow isothermal layer appearing in winter, but can be caused also by an input of fresh water close to river estuaries, or by water which is quite cold in surface. The thermocline is a layer in which 
both temperature and sound velocity decrease with depth, while the isothermal layer 
is a layer with constant temperature, where 
the sound speed increases linearly with depth due to the hydrostatic pressure. 

In stratified media, sound waves can be trapped by acoustic ducts 
to propagate horizontally \cite{Bre, Wil}, therefore the scattering of
sound waves by bounded obstacles is much more complicated
than that in homogeneous media.

\begin{figurehere}
     \begin{center}
     \vskip -0.3truecm
           \scalebox{0.6}{\includegraphics{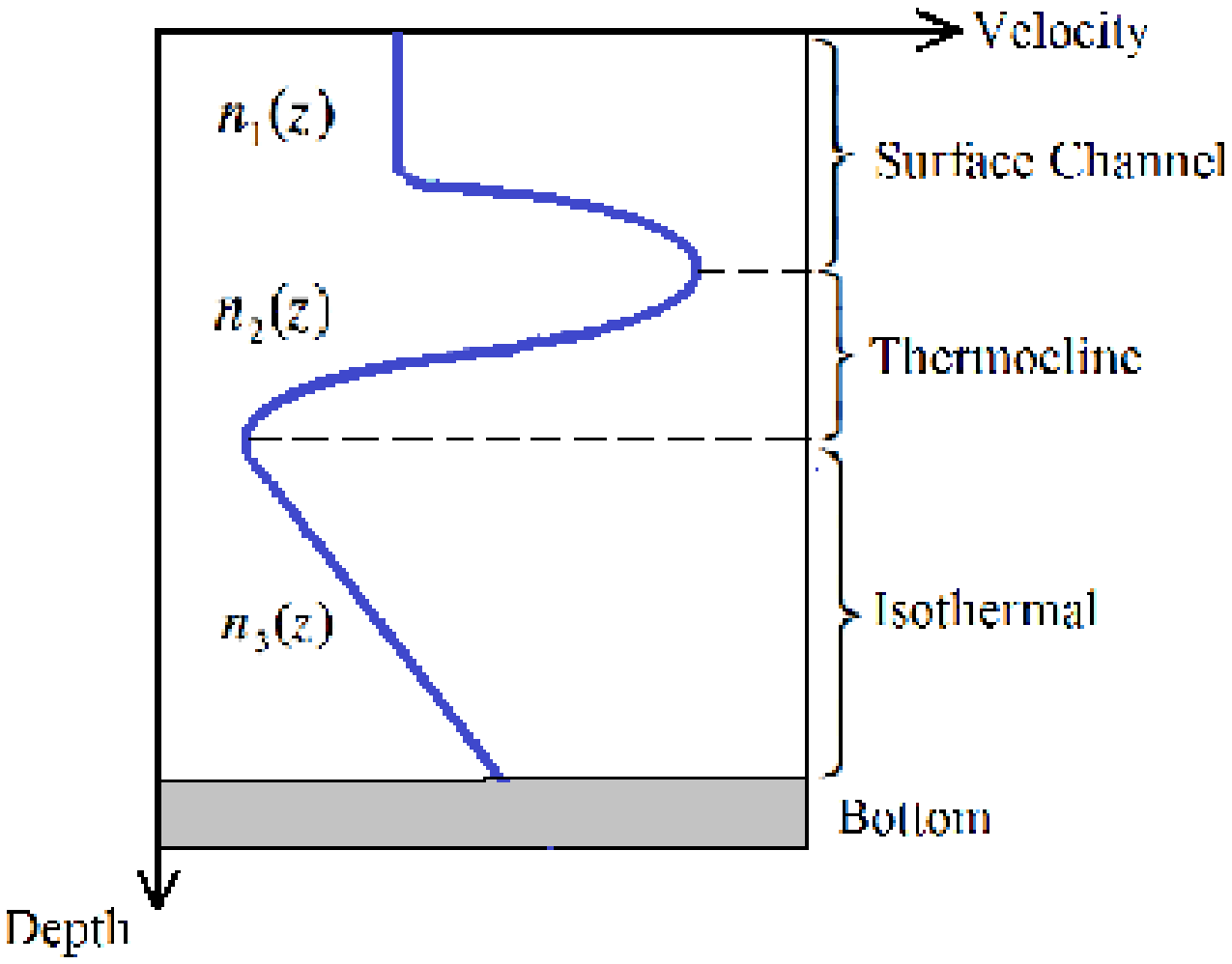}}\\
     \vskip -0.3truecm
    \caption{\label{fig:SVP}\emph{The demonstration of a typical sound velocity profile.}}
     \end{center}
 \end{figurehere}
 
The wave propagation in stratified media has been well studied, and we have 
a general understanding and a detailed description of how sound travels in the ocean. The theory could be applied to make quantitative computations of the sound field induced by an artificial source. Some mathematical modeling 
was studied for the wave propagation in stratified media \cite{GX1, GX2, GX},  
several methods were developed 
for solving the modeling problem and very interesting results were observed \cite{GX3, GX4, GXT}. 
Other studies can be found in \cite{AGL1, PEE, P}.

However, contrary to the large number of techniques developed for simulating wave propagations, 
few studies exist for inverse underwater sound problems in three dimensions, and 
nearly all existing computations are for the reconstruction of 2D obstacles.  
Imaging a scatterer in a waveguide is much more challenging than in the free space. As mentioned in \cite{CH}, due to the presence of two parallel infinite boundaries of the waveguide, only a finite number
of wave modes can propagate at long distance, while the other modes decay exponentially \cite{XU1}. 
Even though many methods have been developed for inverse acoustic scattering problems, 
e.g.,  the MUSIC-type algorithm to locate small inclusions \cite{AIK}, the generalized dual space method \cite{XML}, 
the linear sampling method \cite{AGL2, BL, SZ}, the Kirchhoff migration-based method \cite{RLA}, 
and the direct sampling method  \cite{CILZ, LXZ} to reconstruct obstacles, 
they are all for acoustic waves in homogeneous waveguides. 
More investigations are desired for inverse problems in 3D stratified waveguides. 

%


The detection of black boxes of airplanes attracts much more attention recently since the air crash of MH370. 
A multilevel sampling method is proposed in this work to detect the location of a source in a stratified ocean waveguide with local non-stratified perturbations, which would provide an effective and simple alternative to estimate the positions of the black boxes.  
We shall study a full 3D underwater stratified model, present a mathematical analysis and an effective algorithm to locate the sound source in the complicate stratified ocean environment. 
A time-harmonic point source is assumed and located in the essentially stratified ocean waveguide, which 
is bounded above by a horizontal planar free surface where the acoustic pressure $p$ vanishes, i.e., $p=0$, 
and below by a horizontal planar bottom on which the normal derivative of the acoustic pressure $p$ vanishes, 
namely, $\frac{\partial p}{\partial \nu}=0$.  This set-up and its studies can be found in the lecture notes of physics 
\cite{JJ} and many other works, see, e.g., \cite{GX, GX4, GXT, J, LXZ}.
Our new method can be applied to other stratified oceans with appropriate modifications.

The paper is organized as follows. In section \ref{AM}, the acoustic model and Green's function of the 
3D three-layered waveguide are stated, along with some useful notation, properties and identities. In section \ref{DS}, we 
formulate the direct scattering problem of 3D three-layered waveguide with a known inclusion, and an iterative method is proposed to solve the direct scattering problem of the perturbed waveguide. Section \ref{MSM} describes a novel multilevel sampling method using partial scattered data, and section \ref{NS} provides extensive numerical experiments to evaluate the performance of the multilevel sampling method. Finally, some concluding remarks are presented in section \ref{CR}.

\section{Acoustic model in a three-layered waveguide and Green's function}\label{AM}
In this section, we describe the direct scattering problem of our interest. Consider a three dimensional waveguide $\Omega$ in $\mathbb{R}_h^3=\mathbb{R}^2\times(0,h)$, where $h>0$ is the depth of the ocean. The third coordinate axis is singled out as the one orthogonal to the waveguide, so we shall write
\[x=(x_1,x_2,x_3)^\top=(\tilde{x},x_3)^\top \quad\forall x\in\mathbb{R}_h^3.\]
The upper and lower boundaries of the waveguide are represented by 
\[\Gamma^+:=\{x\in\mathbb{R}^3;\;x_3=h\}\quad\text{and}\quad\Gamma^-:=\{x\in\mathbb{R}^3;\;x_3=0\}.\]
There are three layers lying inside the waveguide, given respectively by 
\begin{eqnarray*}
M_1&=&\{x\in\mathbb{R}^3;\;0<x_3<d_1\},\\[2mm]
M_2&=&\{x\in\mathbb{R}^3;\;d_1<x_3<d_2\},\\[2mm]
M_3&=&\{x\in\mathbb{R}^3;\;d_2<x_3<h\},
\end{eqnarray*}
where $d_1$ and $d_2$ are positive constants, and $h>d_2>d_1$.
Moreover, the two interfaces between each two layers inside the waveguide are written as 
\begin{eqnarray*}
\Gamma_1=\{x\in\mathbb{R}^3;\;x_3=d_1\} \quad \mbox{and} \quad 
\Gamma_2=\{x\in\mathbb{R}^3;\;x_3=d_2\}.
\end{eqnarray*}
The detailed configuration of the waveguide is shown in Figure \ref{fig:model}.
\begin{figurehere}
     \begin{center}
           \scalebox{0.6}{\includegraphics{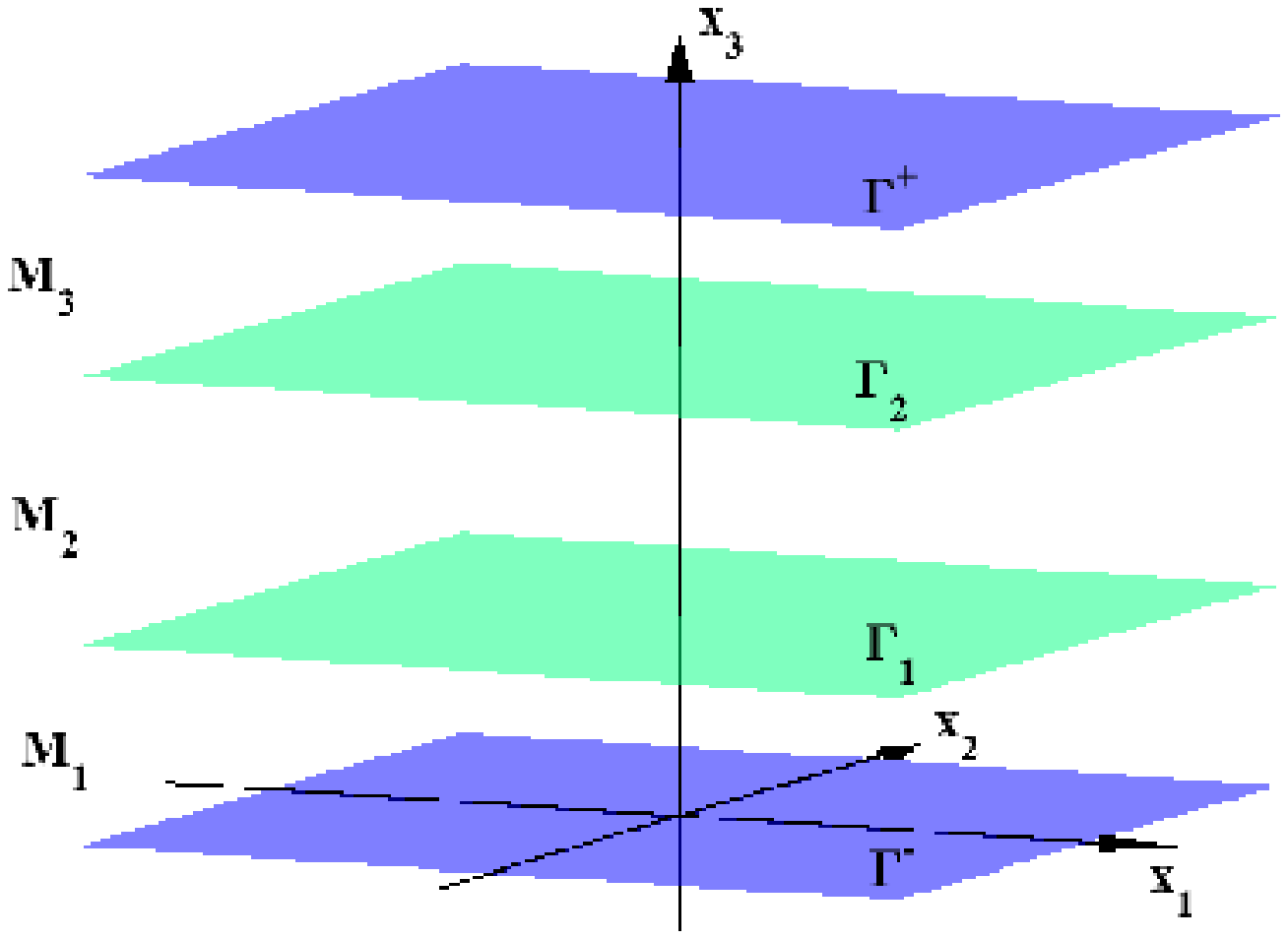}}\\
     \vskip -0.3truecm
    \caption{\label{fig:model}\emph{Geometrical illustration of the waveguide.}}
     \end{center}
 \end{figurehere}

We shall write the total sound pressure field without convection in the 3-layered waveguide 
with a wave-penetrable medium as 
\begin{equation}
p(x)=\left\{
\begin{aligned}
&p_1(x), \quad x\in M_1\,,\\[2mm]
&p_2(x), \quad x\in M_2\,,\\[2mm]
&p_3(x), \quad x\in M_3\,, 
\end{aligned}
\right.
\end{equation} 
which is related to a fluid velocity $v$ by the relation 
\begin{equation}
\nabla p=-\rho \frac{\partial v}{\partial t},
\end{equation}
where $\rho$ is the medium density.

Consider a source point $x_s=(\tilde{x}_s, x_{s3})$, located in the layer $M_1$, then 
the corresponding propagating wave, is modeled by the following Helmholtz system (see \cite{GX4} for a two-layered model), 
whose solution is the so-called Green's function associated with 
the current three-layered stratified waveguide and will be fundamental 
to our source reconstruction in the subsequent sections:
\begin{equation}\label{eq:hel}
\left\{
\begin{aligned}
&\Delta p_1+k_1^2n^2_1p_1=-\delta(\tilde{x}-\tilde{x}_s)\delta(x_3-x_{s3})\quad  \text{in}\quad M_1,\\[2mm]
&\Delta p_2+k_2^2n^2_2p_2=0\quad  \text{in}\quad M_2,\\[2mm]
&\Delta p_3+k_3^2n^2_3p_3=0\quad  \text{in}\quad M_3,\\[2mm]
&\rho_1 p_1=\rho_2p_2 \quad  \text{on}\quad \Gamma_1,\\[2mm]
&\frac{\partial p_1}{\partial\nu}=\frac{\partial p_2}{\partial \nu} \quad  \text{on}\quad \Gamma_1,\\[2mm]
&\rho_2p_2=\rho_3p_3 \quad  \text{on}\quad \Gamma_2,\\[2mm]
&\frac{\partial p_2}{\partial\nu}=\frac{\partial p_3}{\partial \nu} \quad  \text{on}\quad \Gamma_2,\\[2mm]
&p_1=0 \quad \text{on}\quad \Gamma^-,\\[2mm]
&\frac{\partial p_3}{\partial x_3}=0 \quad \text{on}\quad \Gamma^+,
\end{aligned}
\right.
\end{equation} 
where $p_i$ and $k_i$ $(i=1, 2, 3)$ represent respectively the sound pressure and the wavenumber 
in the layer $M_i$, $n_i$ and $\rho_i$ are respectively the refractive index and density in $M_i$, 
and $\nu$ is the normal vector. As we may see, 
the model \eqref{eq:hel} is a 3D model with a three-layered refraction index. This is clearly 
more realistic than the model in \cite{GX4}, which is for two dimensions with only a two-layered refraction index. 
We know that the three-layered sound velocity profile may lead to underwater sound channel in the ocean, 
while the two-layered model does not reflect this important feature of the underwater sound. 
Moreover, the numerical computation of the 3D Green's function is much more complex and difficult to compute
than the 2D one.

A radiation condition is needed to ensure the uniqueness of the system \eqref{eq:hel}, and its specific form 
will be stated at the end of this section after we derive the Green's function for the outgoing wave.

We shall write the outgoing Green's function 
%
as $G(\tilde{x},x_3,\tilde{x}_s,x_{s3})$, and it satisfies $G(\tilde{x},x_3,\tilde{x}_s,x_{s3})=G(|\tilde{x}-\tilde{x}_s|,x_3,x_{s3})$. Let $r=|\tilde{x}-\tilde{x}_s|$, we can represent the Green's function by the Hankel transform \cite{AK}, 
\begin{equation} \label{Green}
G(r,x_3,x_{s3})=\frac{1}{2\pi}\int_0^\infty J_0(\xi r)\widehat{G}(\xi,x_3,x_{s3})\xi d\xi,
\end{equation}
where $\widehat{G}(\xi,x_3,x_{s3})$ is given by 
\begin{equation}
\widehat{G}(\xi,x_3,x_{s3})=
\left\{\begin{aligned}
&\widehat{G}_1(\xi,x_3,x_{s3}),\; 0\leq x_3< d_1,\\[2mm]
&\widehat{G}_2(\xi,x_3,x_{s3}),\; d_1< x_3< d_2,\\[2mm]
&\widehat{G}_3(\xi,x_3,x_{s3}),\; d_2< x_3< h,
\end{aligned}
\right.
\end{equation}
and $\widehat{G}_1$, $\widehat{G}_2$ and $\widehat{G}_3$ solve the following system
\begin{numcases}{}
\frac{\partial^2 \widehat{G}_1}{\partial x_3^2}+\tau_1^2\widehat{G}_1=-\delta(\tilde{x}-\tilde{x}_s)\delta(x_3-x_{s3}),\; 0<x_3<d_1, \label{eq:G1}\\[2mm]
\frac{\partial^2 \widehat{G}_2}{\partial x_3^2}+\tau_2^2\widehat{G}_2=0,\; d_1<x_3<d_2, \label{eq:G2}\\[2mm]
\frac{\partial^2 \widehat{G}_3}{\partial x_3^2}+\tau_3^2\widehat{G}_3=0,\; d_2<x_3<h, \label{eq:G3}\\[2mm]
\rho_1\widehat{G}_1(\xi,d_1)=\rho_2\widehat{G}_2(\xi,d_1), \label{eq:G1Di}\\[2mm]
\frac{\partial \widehat{G}_1}{\partial x_3}(\xi,d_1)=\frac{\partial \widehat{G}_2}{\partial x_3}(\xi,d_1), \label{eq:G1Nu}\\[2mm]
\rho_2\widehat{G}_2(\xi,d_2)=\rho_3\widehat{G}_3(\xi,d_2),\label{eq:G2Di}\\[2mm]
\frac{\partial \widehat{G}_2}{\partial x_3}(\xi,d_2)=\frac{\partial \widehat{G}_3}{\partial x_3}(\xi,d_2),\label{eq:G2Nu}\\[2mm]
\widehat{G}_1(\xi,0)=0,\label{eq:G10}\\[2mm]
\frac{\partial \widehat{G}_3}{\partial x_3}(\xi,h)=0,\label{eq:G3h}
\end{numcases}
where $\tau_1=\sqrt{k_1^2n_1^2-\xi^2}$, $\tau_2=\sqrt{k_2^2n_2^2-\xi^2}$ and $\tau_3=\sqrt{k_3^2n_3^2-\xi^2}$\,.
We will now construct $\widehat{G}(\xi,x_3,x_{s3})$ in the form
\[\widehat{G}(\xi,x_3,x_{s3})=-\frac{\phi_1(\xi,x_{3<})\phi_2(\xi,x_{3>})}{W(\phi_1,\phi_2)(\xi,x_{s3})},\]
where $\phi_1$ and $\phi_2$ are the solutions of \eqref{eq:G1}-\eqref{eq:G3} with interface conditions \eqref{eq:G1Di}-\eqref{eq:G2Nu}, $\phi_1$ and $\phi_2$ also satisfy \eqref{eq:G10} and \eqref{eq:G3h} respectively, and 

\begin{equation*}
 x_{3<}=\left\{
\begin{aligned}
&x_3, &0<x_3<x_{s3}\\
&x_{s3}, &x_{s3}<x_3<h
\end{aligned}\;\;,
\right.
\quad 
 x_{3>}=\left\{
\begin{aligned}
&x_{s3}, &0<x_3<x_{s3}\\
&x_3, &x_{s3}<x_3<h
\end{aligned}\;\;,
\right.
\end{equation*}
and
\[W(\phi_1,\phi_2)=\begin{vmatrix}
\phi_1 & \phi_2 \\[1mm]
\phi_1' & \phi_2' \\[1mm]
\end{vmatrix}.\]

Then we can express $\phi_1$ and $\phi_2$ explicitly in the following forms
\begin{equation}
 \phi_1(\xi,x_3)=\left\{
\begin{aligned}
&\sin(\tau_1x_3),\quad&0\leq x_3<d_1,\\[2mm]
&A_1\cos(\tau_2x_3)+B_1\sin(\tau_2x_3),\quad&d_1<x_3<d_2,\\[2mm]
&A_2\cos(\tau_3x_3)+B_2\sin(\tau_3x_3),\quad&d_2<x_3<h,
\end{aligned}
\right.
\end{equation}
with 
\begin{equation*}
\left\{
\begin{aligned}
A_1&=\frac{\rho_1}{\rho_2}\sin(\tau_1d_1)\cos(\tau_2d_1)-\frac{\tau_1}{\tau_2}\sin(\tau_2d_1)\cos(\tau_1d_1),\\[2mm]
B_1&=\frac{\rho_1}{\rho_2}\sin(\tau_1d_1)\sin(\tau_2d_1)+\frac{\tau_1}{\tau_2}\cos(\tau_1d_1)\cos(\tau_2d_1),\\[2mm]
A_2&=\frac{\rho_2}{\rho_3}\Big(A_1\cos(\tau_2d_2)+B_1\sin(\tau_2d_2)\Big)\cos(\tau_3d_2)+\frac{\tau_2}{\tau_3}\Big(A_1\sin(\tau_2d_2)-B_1\cos(\tau_2d_2)\Big)\sin(\tau_3d_2),\\[2mm]
B_2&=\frac{\rho_2}{\rho_3}\Big(A_1\cos(\tau_2d_2)+B_1\sin(\tau_2d_2)\Big)\sin(\tau_3d_2)-\frac{\tau_2}{\tau_3}\Big(A_1\sin(\tau_2d_2)-B_1\cos(\tau_2d_2)\Big)\cos(\tau_3d_2), 
\end{aligned}
\right.
\end{equation*}
and 
\begin{equation}
 \phi_2(\xi,x_3)=\left\{
\begin{aligned}
&A_3\cos(\tau_1x_3)+B_3\sin(\tau_1x_3),\quad& 0\leq x_3<d_1,\\[2mm]
&A_4\cos(\tau_2x_3)+B_4\sin(\tau_2x_3),\quad&d_1<x_3<d_2,\\[2mm]
&\cos(\tau_3(h-x_3)), \quad&d_2<x_3< h,
\end{aligned}
\right.
\end{equation}
with 
\begin{equation*}
\left\{
\begin{aligned}
A_3&=\frac{\rho_2}{\rho_1}\Big(A_4\cos(\tau_2d_1)+B_4\sin(\tau_2d_1)\Big)\cos(\tau_1d_1)+\frac{\tau_2}{\tau_1}\Big(A_4\sin(\tau_2d_1)-B_4\cos(\tau_2d_1)\Big)\sin(\tau_1d_1),\\[2mm]
B_3&=\frac{\rho_2}{\rho_1}\Big(A_4\cos(\tau_2d_1)+B_4\sin(\tau_2d_1)\Big)\sin(\tau_1d_1)-\frac{\tau_2}{\tau_1}\Big(A_4\sin(\tau_2d_1)-B_4\cos(\tau_2d_1)\Big)\cos(\tau_1d_1),\\[2mm]
A_4&=\frac{\rho_3}{\rho_2}\cos(\tau_2d_2)\cos(\tau_3(h-d_2))-\frac{\tau_3}{\tau_2}\sin(\tau_2d_2)\sin(\tau_3(h-d_2)),\\[2mm]
B_4&=\frac{\rho_3}{\rho_2}\sin(\tau_2d_2)\cos(\tau_3(h-d_2))+\frac{\tau_3}{\tau_2}\cos(\tau_2d_2)\sin(\tau_3(h-d_2)).
\end{aligned}
\right.
\end{equation*}

As $W(\phi_1,\phi_2)$ does not depend on $x_{s3}$, we write 
$W(\phi_1,\phi_2)$ as $W(\phi_1,\phi_2)(\xi)$ for convenience. Then by a simple calculation, we derive 
\begin{equation}\label{eq:wronski}
\begin{aligned}
W(\phi_1,\phi_2)(\xi)&=\phi_1(\xi,0)\phi'_2(\xi,0)-\phi_2(\xi,0)\phi'_1(\xi,0)\\[2mm]
&=-\phi_2(\xi,0)\phi'_1(\xi,0)=-A_3\tau_1.
\end{aligned}
\end{equation}

Using all the above expressions,  we can further represent the Green's function in \eqref{Green} as
\begin{equation}
G(r,x_3,x_{s3})=-\frac{1}{2\pi}\int_0^\infty J_0(\xi r)\frac{\phi_1(\xi,x_{3<})\phi_2(\xi,x_{3>})}{W(\phi_1,\phi_2)}\xi d\xi.
\end{equation}

The above integral can be considered as a contour integral along the contour $C_1$ 
from the origin to infinity slightly below the real axis out to some large real values of $\xi$ and then along the axis (cf. \cite{AK}). The integral has poles at the zeros of $W(\phi_1,\phi_2)(\xi)$, which are denoted by $\xi_n$. At the zeros, $\phi_1(\xi_n,x_3)$ and $\phi_2(\xi_n,x_3)$ are linearly dependent, say, there exists a constant $c_n$ such that
\[\phi_2(\xi_n,x_3)=c_n\phi_1(\xi_n,x_3).\]
 
Let $C_2$ be $C_1e^{i\pi}$ with the orientation reversed.
From the fact that $H^{(2)}_0(\xi r)=-H^{(1)}_0(\xi e^{i\pi}r)$, and by the residual theorem, we can readily derive 
\begin{eqnarray}
G(r,x_3,x_{s3})&=&-\frac{1}{2\pi}\int_{C_1} J_0(\xi r)\frac{\phi_1(\xi,x_{3<})\phi_2(\xi,x_{3>})}{W(\phi_1,\phi_2)}\xi d\xi\notag\\[2mm]
&=&-\frac{1}{4\pi}\int_{C_1}\Big(H^{(1)}_0(\xi r)+H^{(2)}_0(\xi r)\Big)\frac{\phi_1(\xi,x_{3<})\phi_2(\xi,x_{3>})}{W(\phi_1,\phi_2)}\xi d\xi\notag\\[2mm]
&=&-\frac{1}{4\pi}\int_{C_1+C_2}H^{(1)}_0(\xi r)\frac{\phi_1(\xi,x_{3<})\phi_2(\xi,x_{3>})}{W(\phi_1,\phi_2)}\xi d\xi\notag\\[2mm]
&=&-\frac{1}{4\pi}\cdot 2\pi i\sum_{n=1}^\infty\frac{\phi_1(\xi,x_{3<})\phi_2(\xi,x_{3>})}{\frac{\partial}{\partial \xi}W(\phi_1,\phi_2)\big|_{\xi=\xi_n}}H^{(1)}_0(\xi_n r)\notag\\[2mm]
&=&-\frac{i}{2}\sum_{n=1}^\infty\frac{\phi_1(\xi_n,x_{3<})c_n\phi_1(\xi_n,x_{3>})}{\frac{\partial}{\partial \xi}W(\phi_1,\phi_2)\big|_{\xi=\xi_n}}H^{(1)}_0(\xi_n r)\notag\\[2mm]
&=&-\frac{i}{2}\sum_{n=1}^\infty\frac{\Phi_n(x_3)\Phi_n(x_{s3})}{W_n(\phi_1,\phi_2)}H^{(1)}_0(\xi_n r), \label{eq:Greenfcn}
\end{eqnarray}
where
\begin{eqnarray}
&&\Phi_n(x_3)=\frac{\phi_1(\xi_n,x_3)}{\|\phi_1(\xi_n,\cdot)\|},\quad \|\phi_1(\xi_n,\cdot)\|^2=\langle \phi_1(\xi_n,\cdot),\phi_1(\xi_n,\cdot)\rangle:=\int_0^h\phi^2_1(\xi_n,x_3)dx_3,\label{eq:Phi}\\[2mm]
&&W_n(\phi_1,\phi_2)=\frac{1}{c_n\|\phi_1(\xi_n,\cdot)\|^2}\frac{\partial }{\partial\xi}W(\phi_1,\phi_2)|_{\xi=\xi_n}.\notag
\end{eqnarray}
Now we can formulate the outgoing radiation condition. Assume the refractive index is stratified
for large $|\tilde{x}|$, we have by separation of variables the following mode expansion of the propagating wave \cite{AK, Wil}:
\begin{equation}\label{eq:radicond}
p(\tilde{x},x_3)=\sum_{n=1}^\infty \Phi_n(x_3)u_n(\tilde{x}) \quad \text{for}\;\;\text{large}\;\; |\tilde{x}|,
\end{equation}
where $\Phi_n(x_3)$ is given by \eqref{eq:Phi}, and the mode expansion coefficient $u_n(\tilde{x})$ satisfies 
\begin{equation}
\lim_{|\tilde{x}|\to\infty}\Big(\frac{\partial u_n}{\partial |\tilde{x}|}-i\xi_nu_n\Big)=0\quad \text{for} \;\; n=1,2,\cdots\,.
\end{equation}

For the ease of observation, we show a slice of the wave propagation contour of a point source located at (0,0,50) in a three-layered waveguide $\mathbb{R}_h^3$ in Figure \ref{fig:wavepropagation}(a) and the corresponding 3D display is shown in Figure \ref{fig:wavepropagation}(b). 
The corresponding modes $\phi_1(\xi_n)$ with $n=1,2,\cdots,20$ are shown in Figure \ref{fig:mode}.

\begin{figurehere}
 \hfill{}\includegraphics[clip,width=0.08\textwidth]{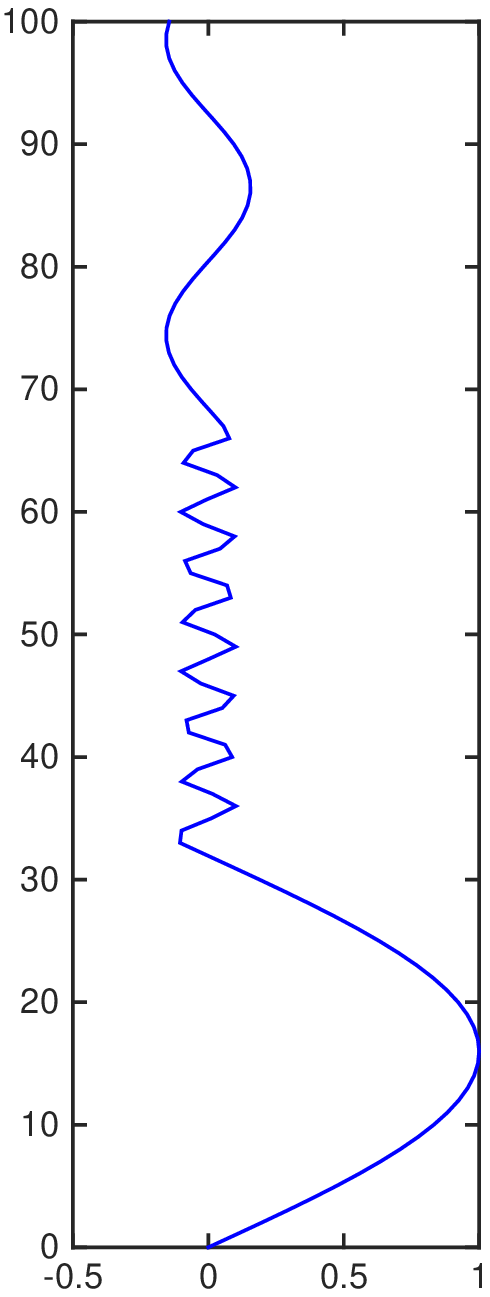}\hfill{}
 \hfill{}\includegraphics[clip,width=0.08\textwidth]{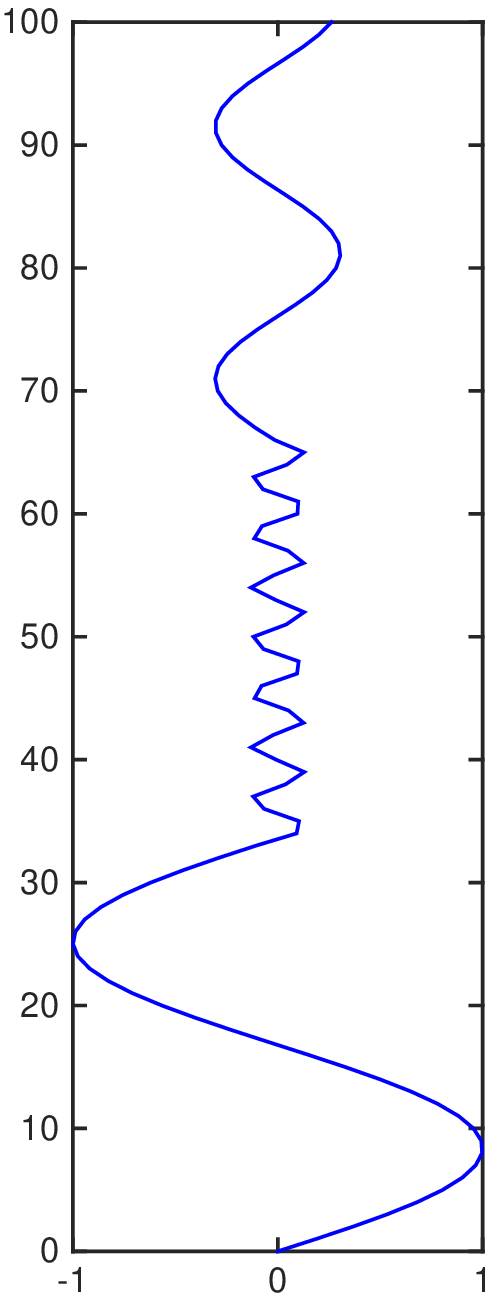}\hfill{}
 \hfill{}\includegraphics[clip,width=0.08\textwidth]{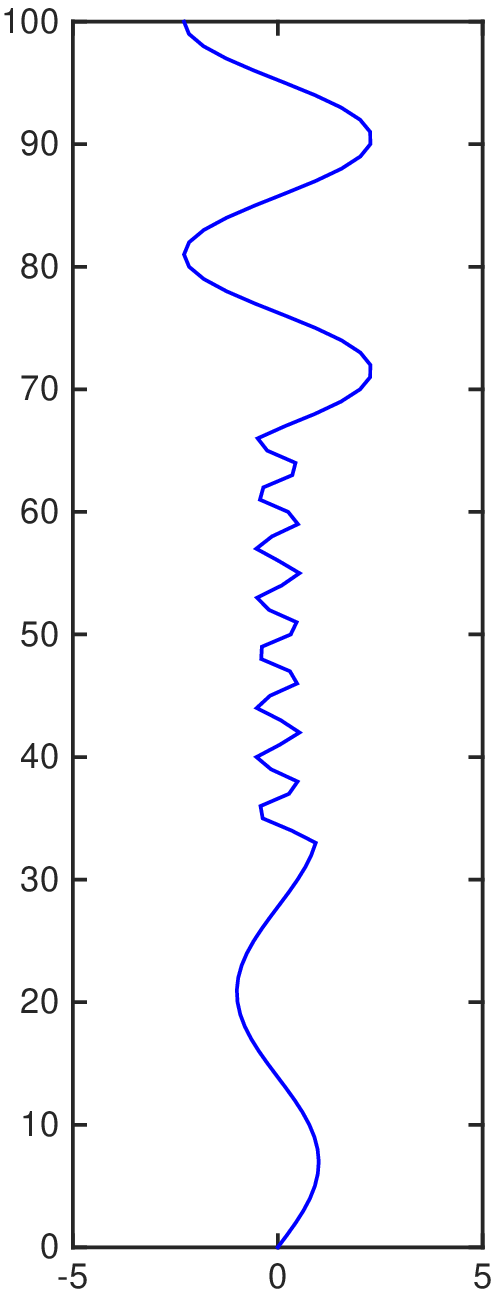}\hfill{}
 \hfill{}\includegraphics[clip,width=0.08\textwidth]{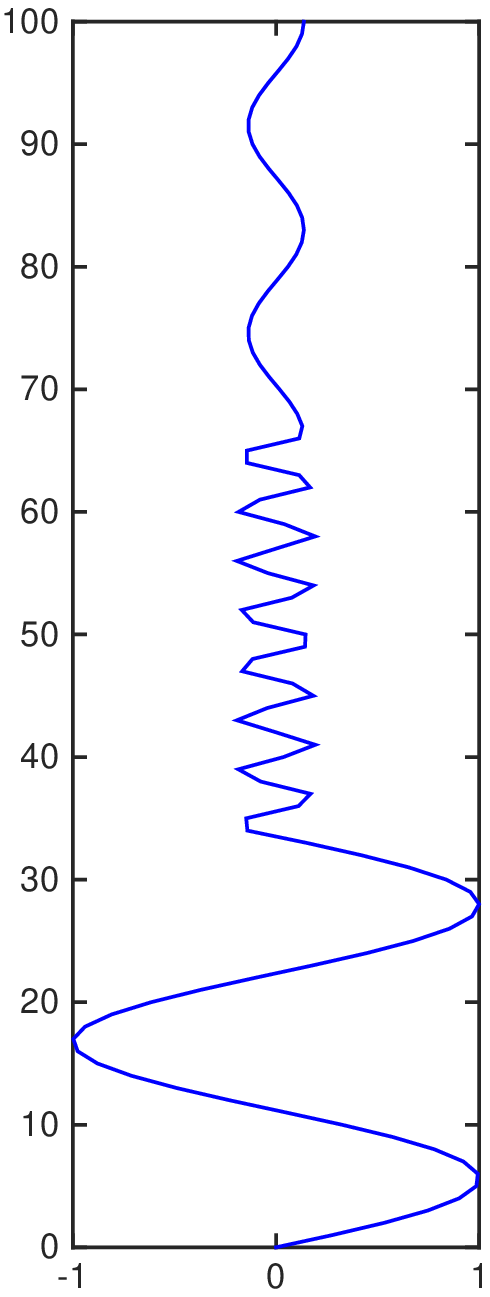}\hfill{}
 \hfill{}\includegraphics[clip,width=0.08\textwidth]{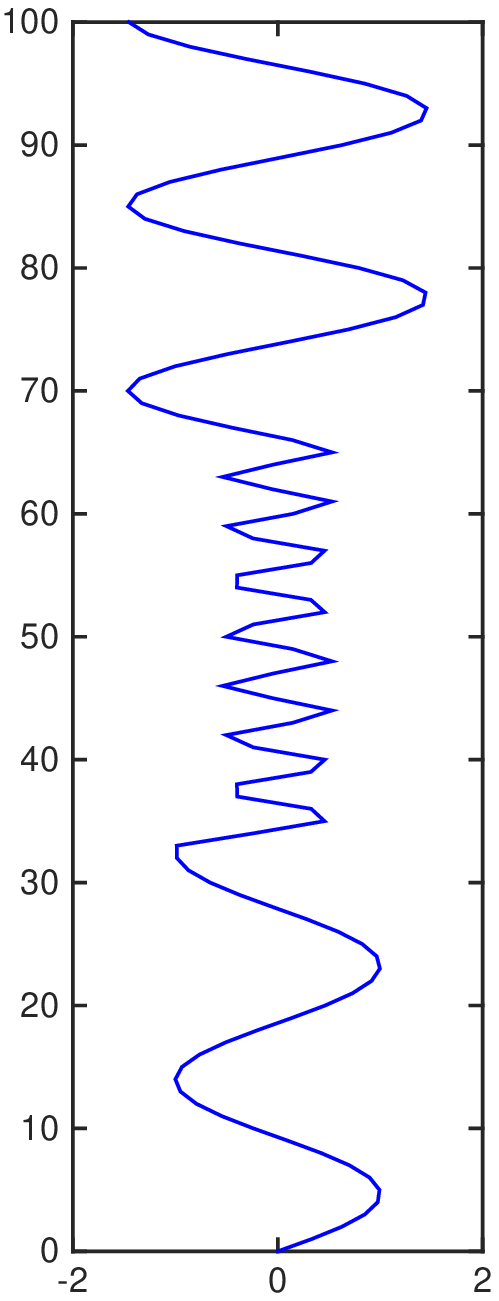}\hfill{}
 \hfill{}\includegraphics[clip,width=0.08\textwidth]{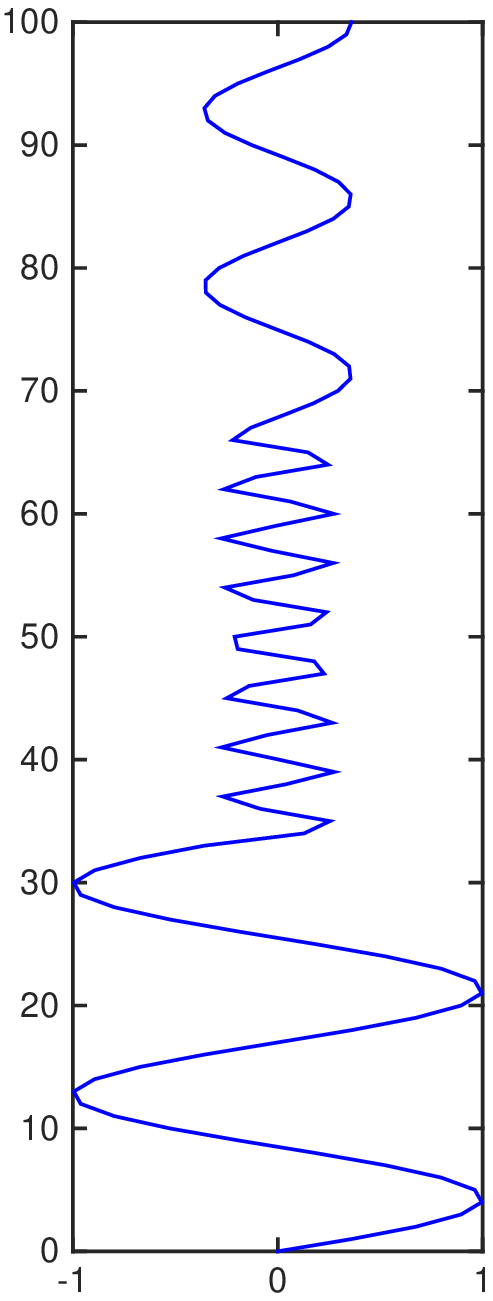}\hfill{}
  \hfill{}\includegraphics[clip,width=0.08\textwidth]{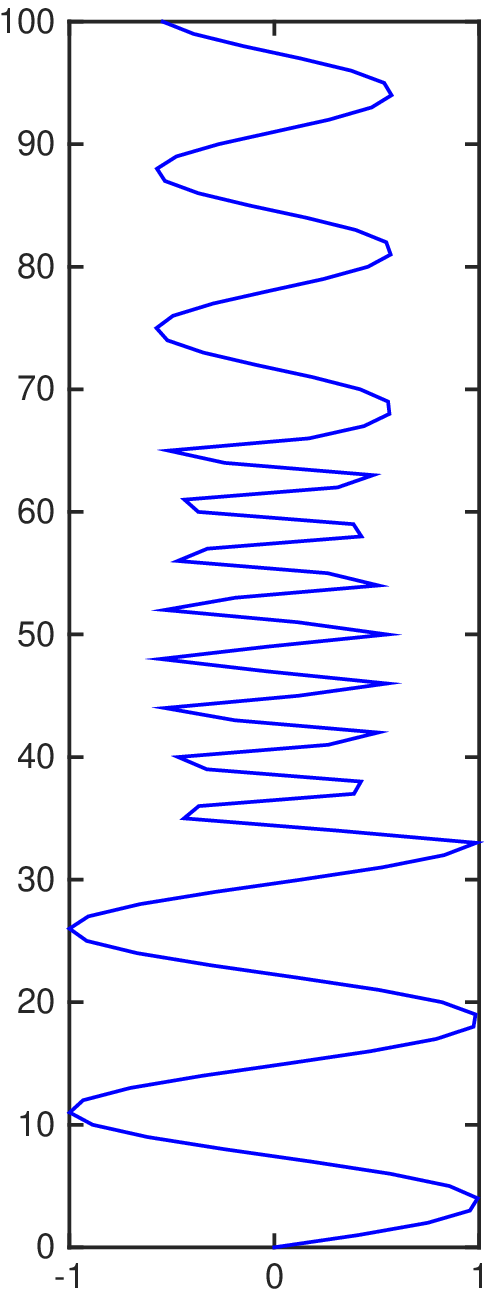}\hfill{}
 \hfill{}\includegraphics[clip,width=0.08\textwidth]{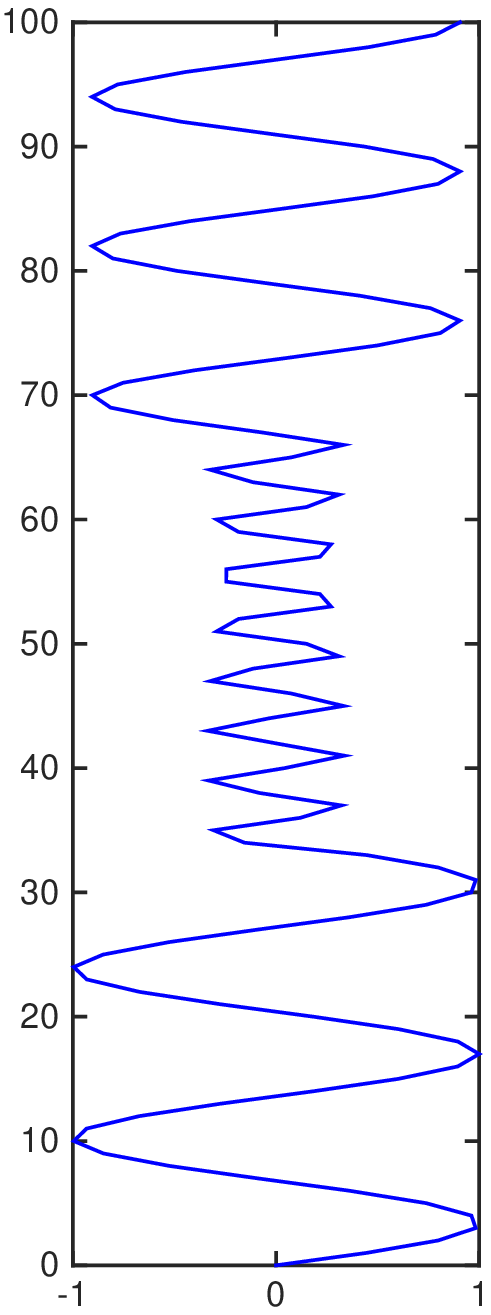}\hfill{}
 \hfill{}\includegraphics[clip,width=0.08\textwidth]{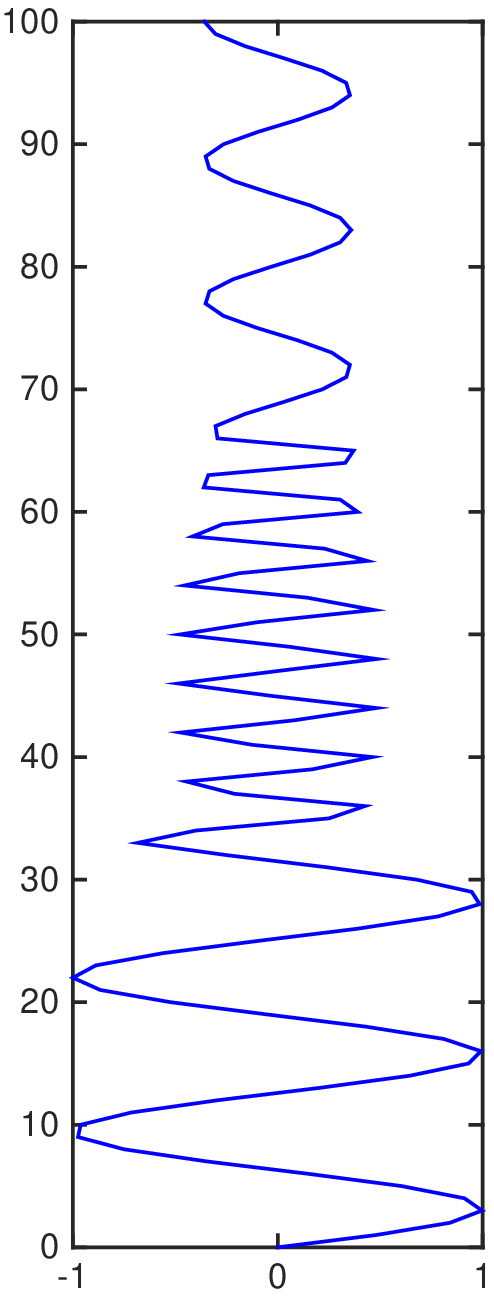}\hfill{}
 \hfill{}\includegraphics[clip,width=0.08\textwidth]{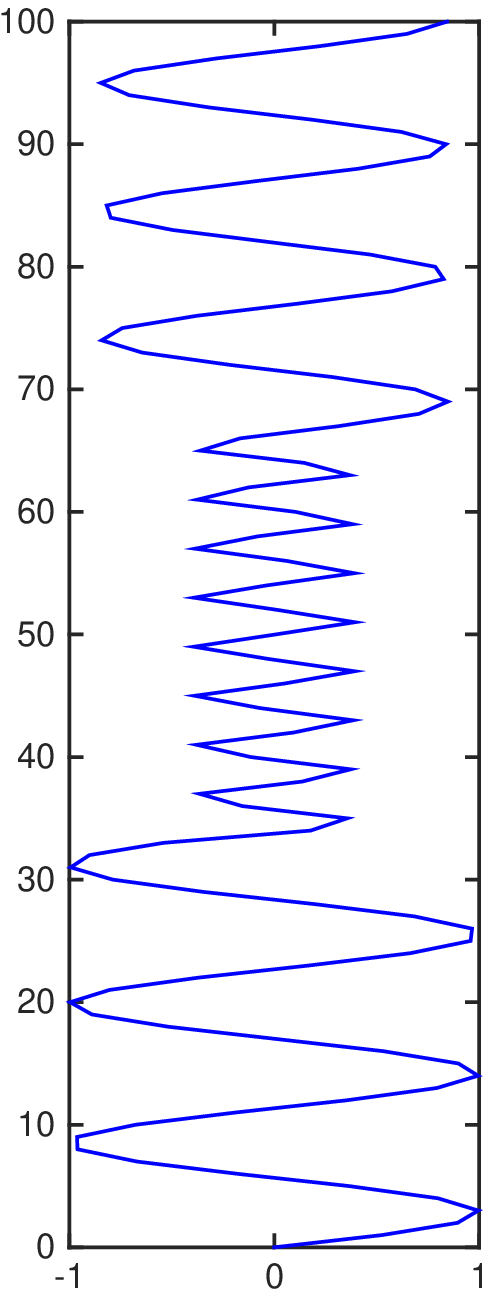}\hfill{}
 
 \hfill{}\includegraphics[clip,width=0.08\textwidth]{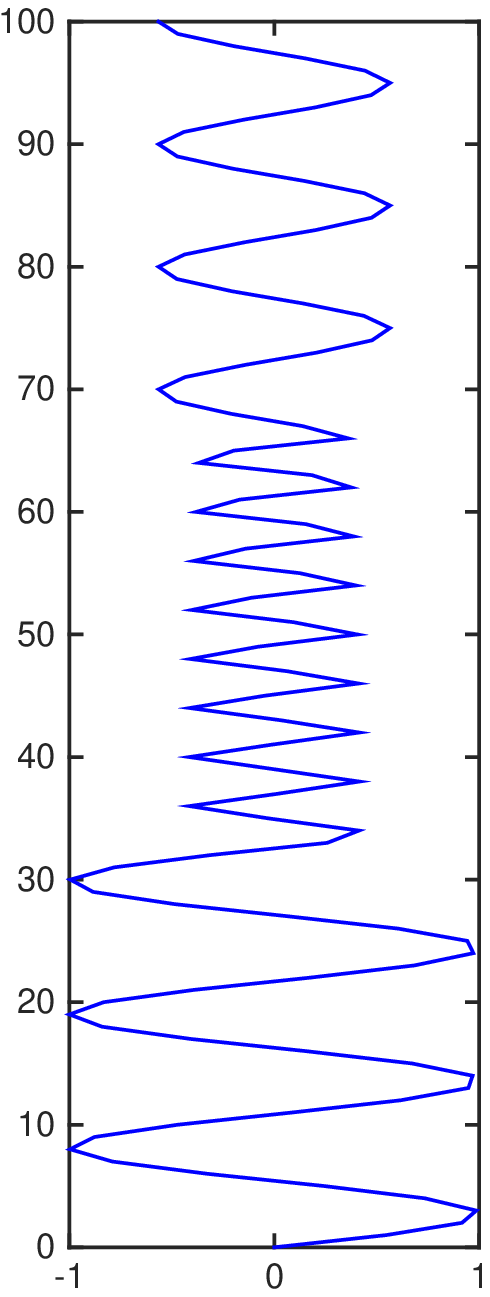}\hfill{}
 \hfill{}\includegraphics[clip,width=0.08\textwidth]{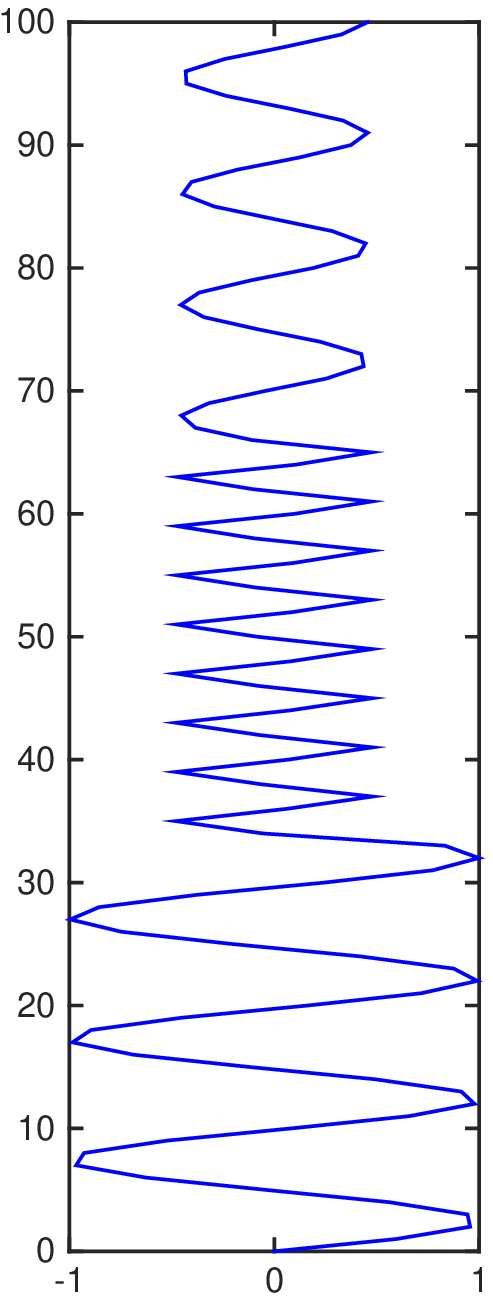}\hfill{}
 \hfill{}\includegraphics[clip,width=0.08\textwidth]{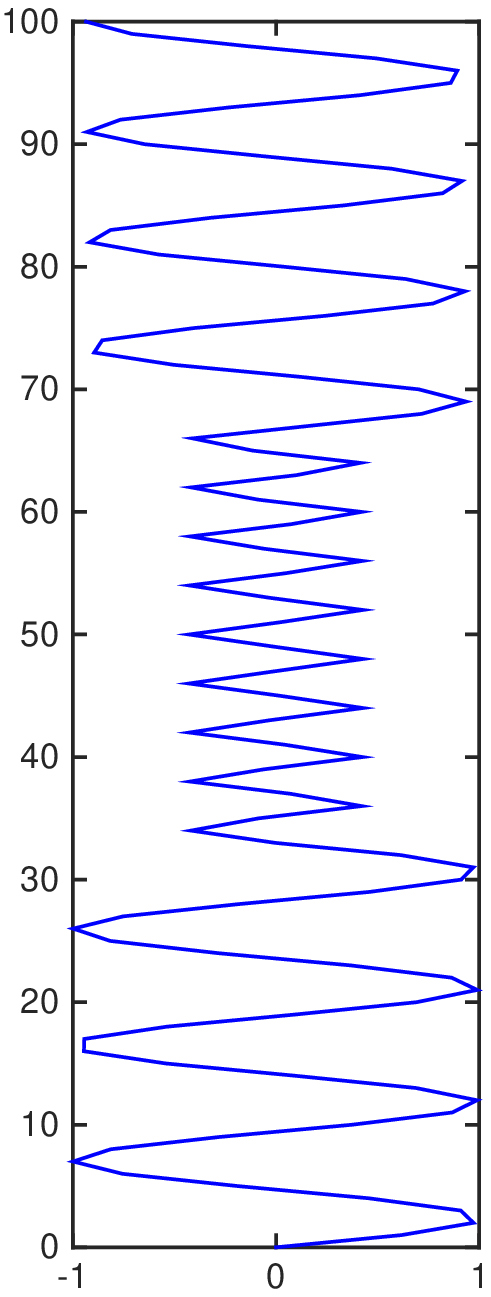}\hfill{}
 \hfill{}\includegraphics[clip,width=0.08\textwidth]{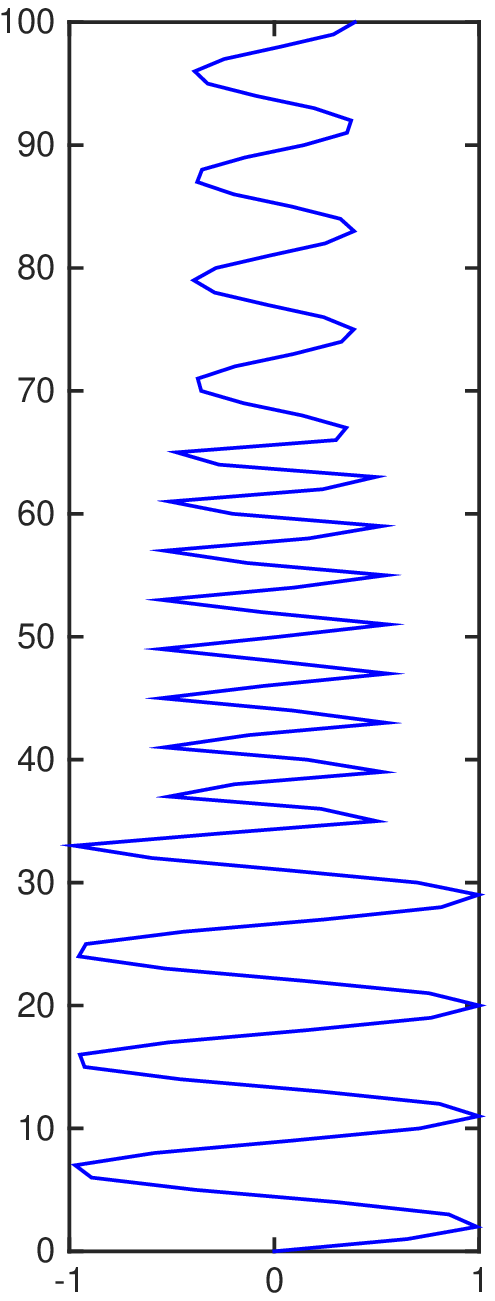}\hfill{}
 \hfill{}\includegraphics[clip,width=0.08\textwidth]{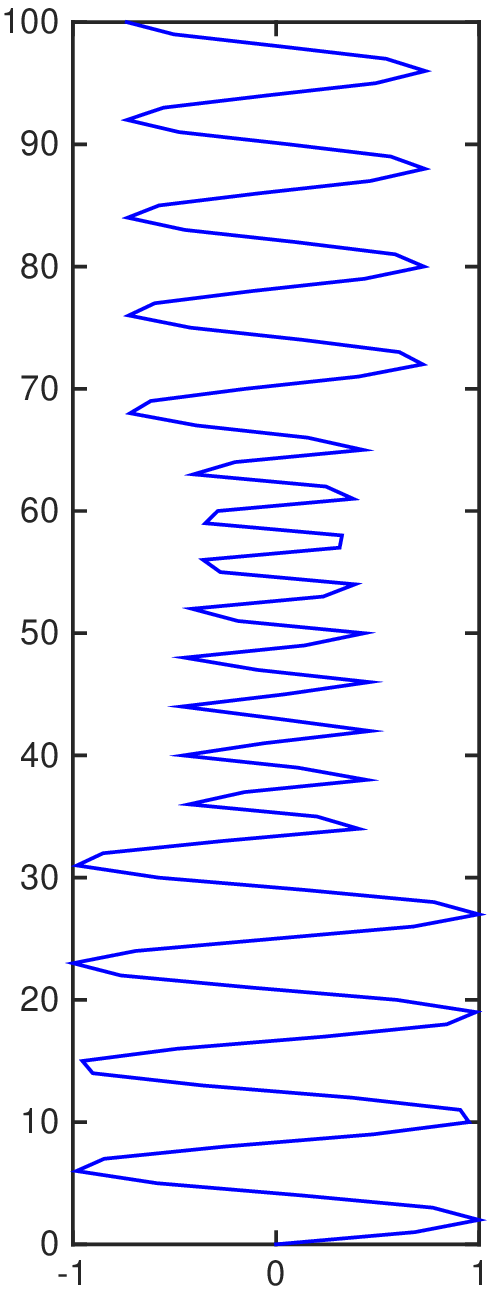}\hfill{}
 \hfill{}\includegraphics[clip,width=0.08\textwidth]{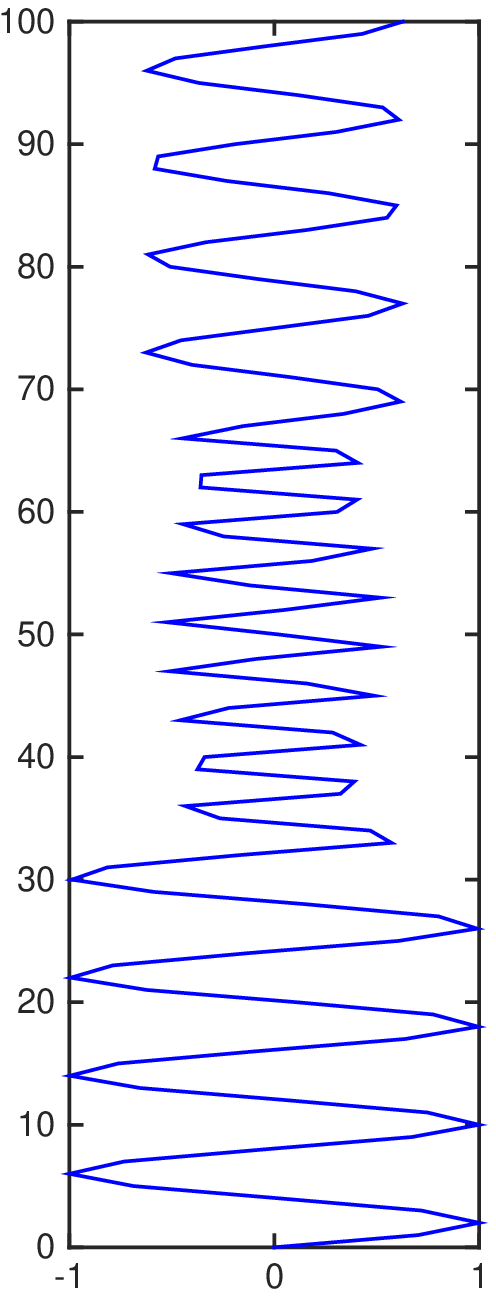}\hfill{}
  \hfill{}\includegraphics[clip,width=0.08\textwidth]{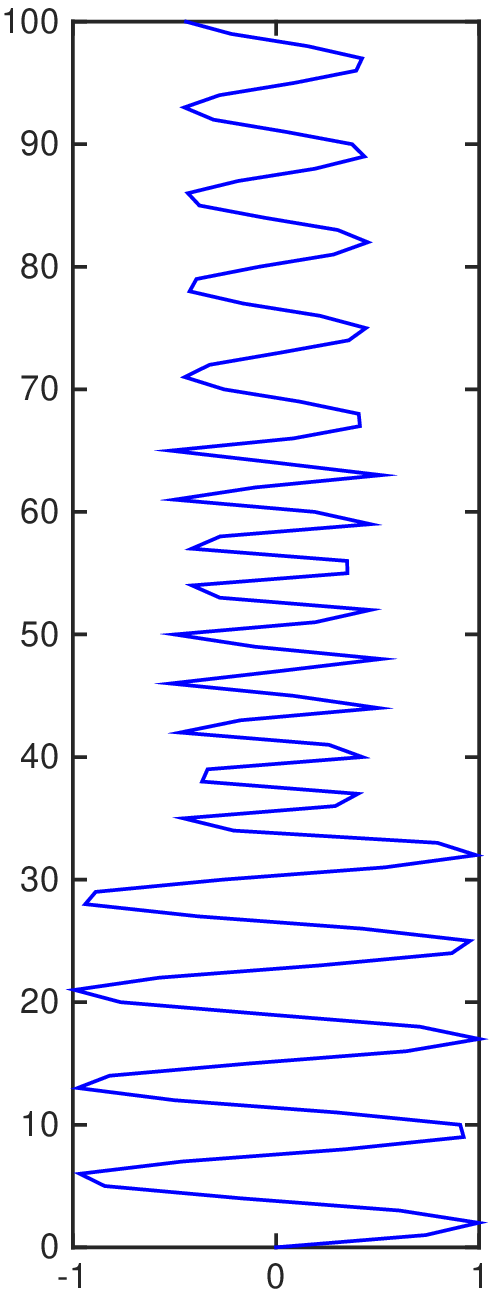}\hfill{}
 \hfill{}\includegraphics[clip,width=0.08\textwidth]{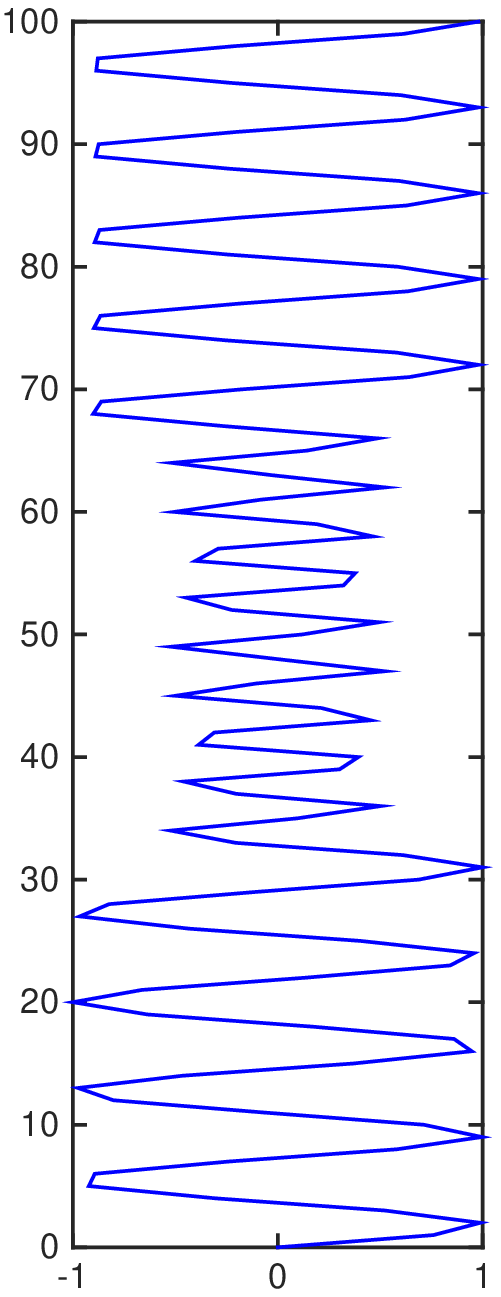}\hfill{}
 \hfill{}\includegraphics[clip,width=0.08\textwidth]{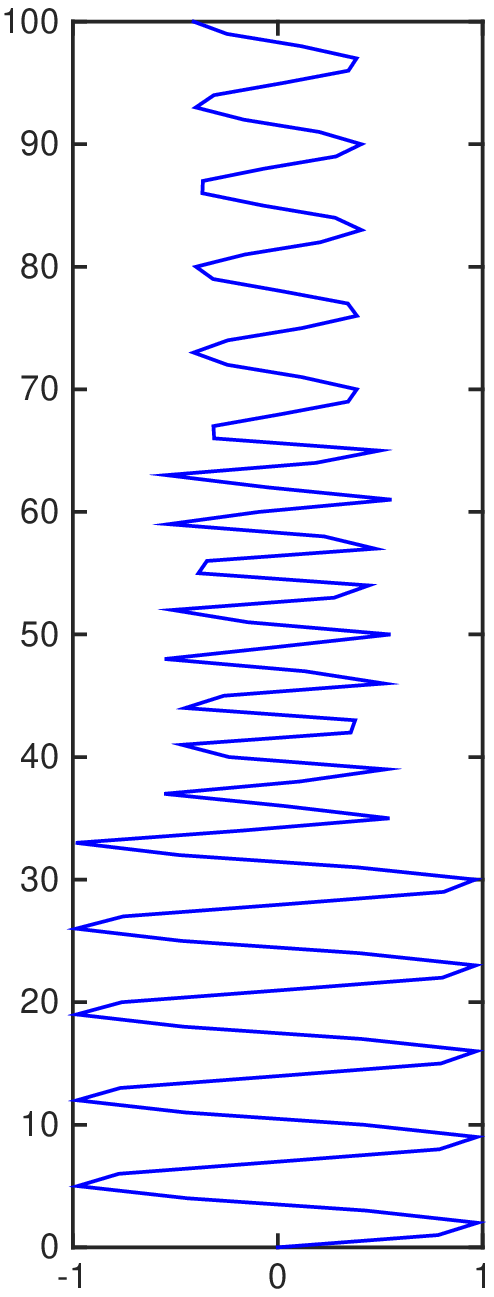}\hfill{}
 \hfill{}\includegraphics[clip,width=0.08\textwidth]{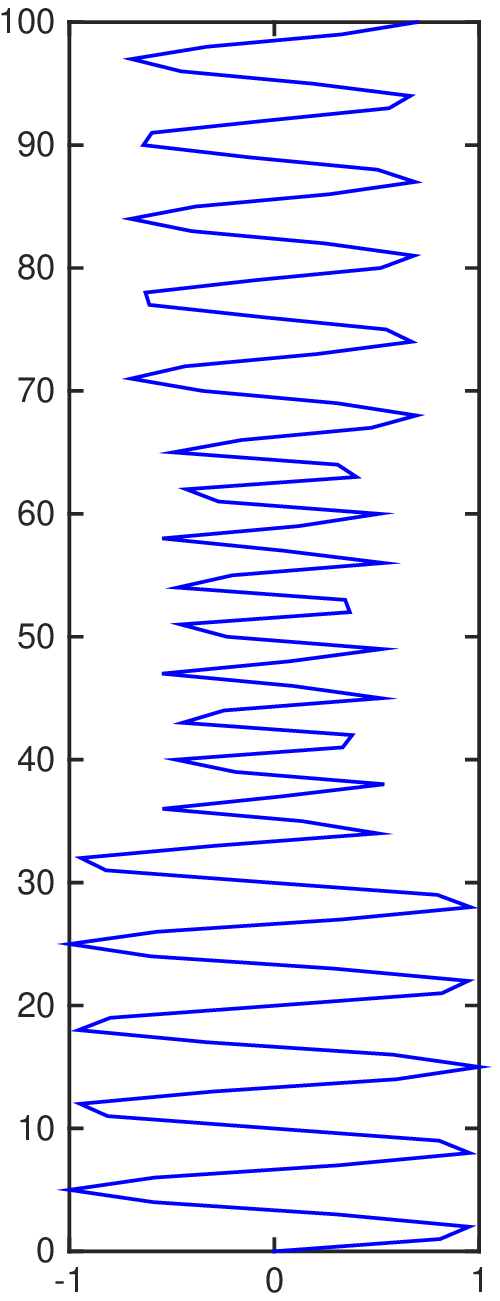}\hfill{}

 \caption{\label{fig:mode} \small{Modes $\phi_1(\xi_n)$ for $n=1,2,\cdots,20$.}}
 \end{figurehere}

\begin{figurehere}
	\hfill{}\includegraphics[clip,width=0.8\textwidth]{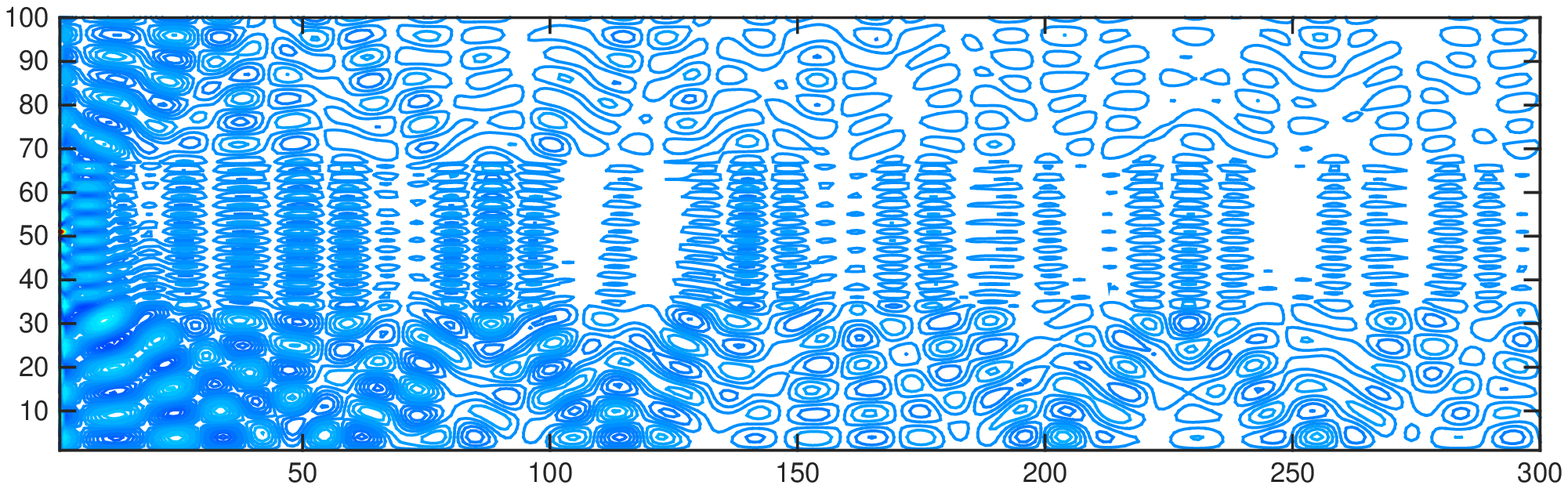}\hfill{}
	
	\hfill{}(a)\hfill{}
	
 	\hfill{}\includegraphics[clip,width=0.8\textwidth]{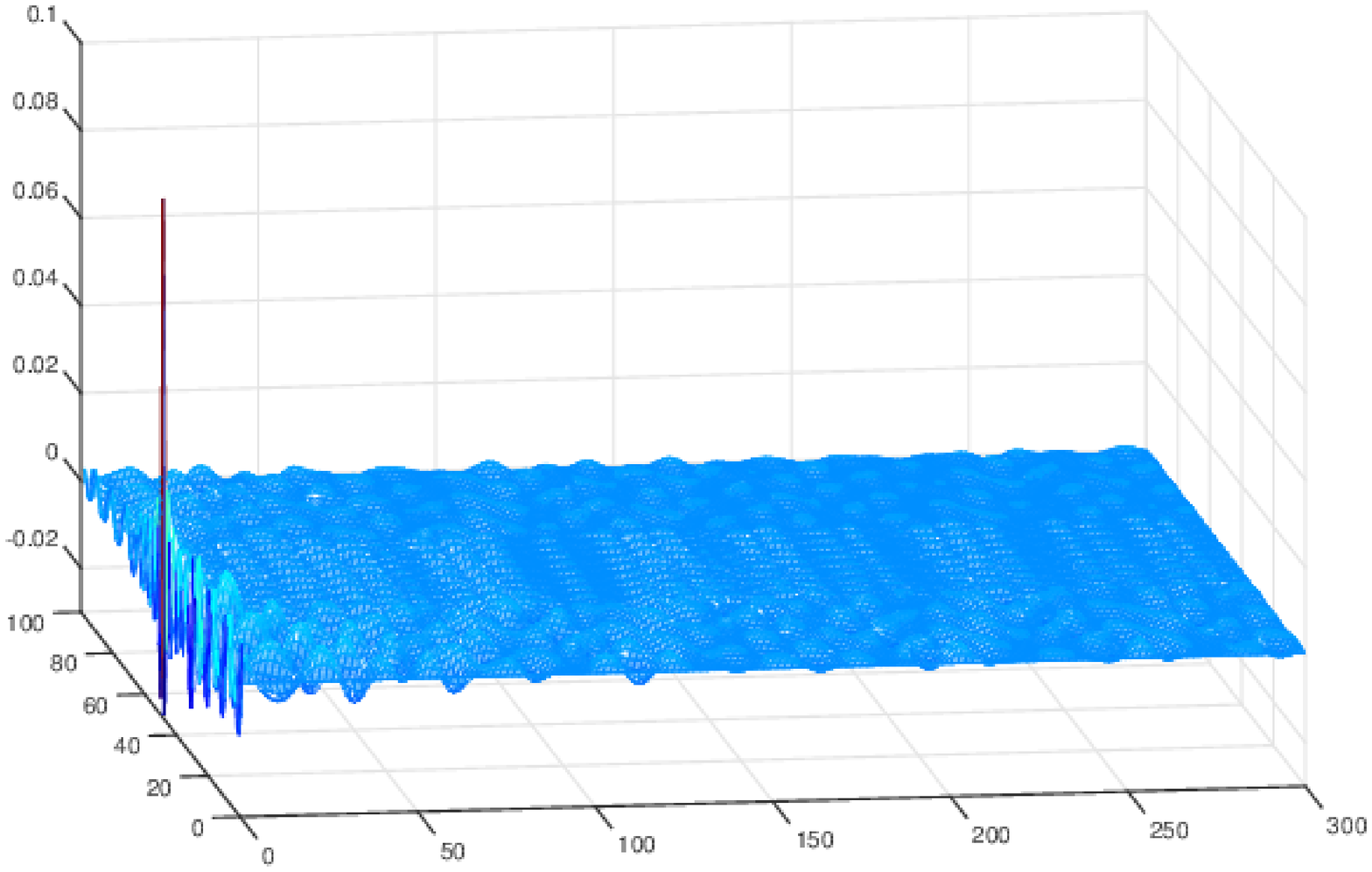}\hfill{}
	
 	 \hfill{}(b)\hfill{}
    \caption{\label{fig:wavepropagation}\emph{Wave propagation of a point source in a three-layered waveguide.}}
 \end{figurehere}

\section{Point source propagating sound wave in a three-layered waveguide with an inhomogeneous inclusion}\label{DS}

\ms
In this section, we shall consider the point source propagating in the three-layered waveguide $\mathbb{R}^3_h$, 
where there is one inhomogeneous inclusions $\Omega$ located 
inside, and 
the part of the waveguide that is not occupied by $\Omega$ is connected. 
The inclusion $\Omega$ may be an object that is known to exist in the waveguide region 
from some earlier exploration or construction, or represents a possible local perturbation of the refraction index 
that may be caused by some practical reasons.  
This three-dimensional three-layered waveguide model is clearly more realistic 
than the one considered in the literature; see \cite{GX4}, where only a two-layered waveguide in two dimesions
was considered. 
For pratical applications, this inhomogeneous inclusion $\Omega$ may have different geometric shapes and may 
be located in any of the three layers of the waveguide. For the sake of definiteness, 
we assume $\Omega$ lies in the layer 
$M_2$, and the geometry of the entire physical waveguide is shown in Figure \ref{fig:model2}. 
Corresponding to this three-layered stratified waveguide, the acoustic pressure field without convection is expressed as 
\begin{equation}
p(x)=\left\{
\begin{aligned}
&p_1(x), \quad & x&\in M_1\\[2mm]
&p_2(x), \quad & x&\in M_2\backslash \overline{\Omega}\\[2mm]
&p_3(x), \quad & x&\in M_3\\[2mm]
&p_4(x), \quad & x&\in \Omega
\end{aligned}
\right.
\end{equation}
which satisfies the outgoing radiation condition \eqref{eq:radicond} and the following 
piecewise Helmholtz system:
\begin{numcases}{}
\Delta p_1+k_1^2n^2_1p_1=-\delta(\tilde{x}-\tilde{x}_s)\delta(x_3-x_{s3})\quad  \text{in}\quad M_1,\label{eq:u1}\\[2mm]
\Delta p_2+k_2^2n^2_2p_2=0\quad  \text{in}\quad M_2\backslash \overline{\Omega},\label{eq:u2}\\[2mm]
\Delta p_3+k_3^2n^2_3p_3=0\quad  \text{in}\quad M_3,\label{eq:u3}\\[2mm]
\Delta p_4+k_4^2n^2_4(x)p_4=0\quad  \text{in}\quad \Omega,\label{eq:u4}\\[2mm]
\rho_1p_1=\rho_2p_2 \quad  \text{on}\quad \Gamma_1,\\[2mm]
\frac{\partial p_1}{\partial\nu}=\frac{\partial p_2}{\partial \nu} \quad  \text{on}\quad \Gamma_1,\\[2mm]
\rho_2p_2=\rho_3p_3 \quad  \text{on}\quad \Gamma_2,\\[2mm]
\frac{\partial p_2}{\partial\nu}=\frac{\partial p_3}{\partial \nu} \quad  \text{on}\quad \Gamma_2,\\[2mm]
p_1=0 \quad \text{on}\quad \Gamma^-,\\[2mm]
\frac{\partial p_3}{\partial x_3}=0 \quad \text{on}\quad \Gamma^+,\\[2mm]
\rho_2p_2=\rho_4p_4\quad \text{on}\quad \partial\Omega,\label{eq:u2u4Dir}\\[2mm]
\frac{\partial p_2}{\partial \nu}=\frac{\partial p_4}{\partial\nu}\quad \text{on}\quad \partial\Omega.\label{eq:u2u4Num}
\end{numcases}

For the ease of exposition, we write 
\begin{equation}\label{eq:q}
q^o(x)=\left\{
\begin{aligned}
&k_1n_1,&x\in M_1,\\[2mm]
&k_2n_2,&x\in M_2,\\[2mm]
&k_3n_3,&x\in M_3,
\end{aligned}
\right.
\qquad
\text{and}
\qquad
q(x)=\left\{
\begin{aligned}
&k_1n_1, &x\in& M_1,\\[2mm]
&k_2n_2, &x\in& M_2\backslash\Omega,\\[2mm]
&k_3n_3, &x\in& M_3,\\[2mm]
&k_4n_4(x), &x\in&\Omega.
\end{aligned}
\right.
\end{equation}

\begin{figurehere}
     \begin{center}
           \scalebox{0.6}{\includegraphics{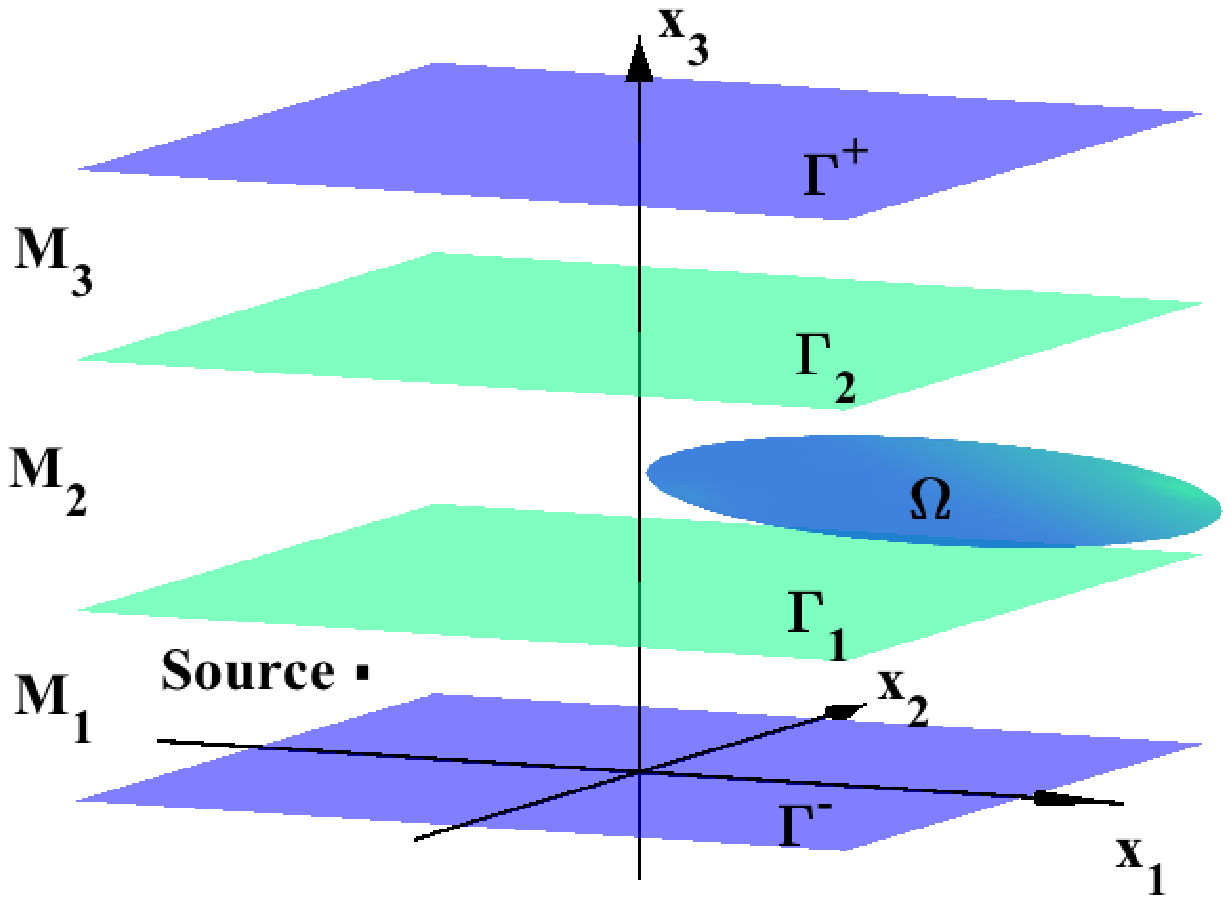}}\\
    \caption{\label{fig:model2}\emph{Illustration of an inhomogeneous inclusion embedded in the waveguide.}}
     \end{center}
 \end{figurehere}
 
Then we can rewrite \eqref{eq:u1}-\eqref{eq:u4} as
\[\Delta p+ (q^o)^2p=\Big[(q^o)^2-q^2\Big]p-\delta(\tilde{x}-\tilde{x}_s)\delta(x_3-x_{s3})\;\;\;\text{a.e.\;\;in\;\;}\;\mathbb{R}_h^3\,,\]
which can be simplified as
\begin{equation}
\label{eq:HelNew}
\Delta p+ (q^o)^2p=\tilde{q}p-\delta(\tilde{x}-\tilde{x}_s)\delta(x_3-x_{s3})\;\;\;\text{a.e.\;\;in\;\;}\;\mathbb{R}_h^3\,,
\end{equation}
where $\tilde{q}$ is given by 
\[\tilde{q}:=(q^o)^2-q^2=\left\{
\begin{aligned}
&0\,, &x\not\in\Omega\,,\\[2mm]
&k_2^2n^2_2-k_4^2n^2_4(x)\,, &x\in\Omega\,.
\end{aligned}
\right.\]

Assume $G(\tilde{\xi},\xi_3;x)$ is the Green's function of the three-layered waveguide with the acoustic point source 
situated at $(\tilde{\xi},\xi_3)$. Multiplying both sides of \eqref{eq:HelNew} by $G(\tilde{\xi},\xi_3;x)$ and integrating over $\Omega$, we have
\begin{eqnarray}
&&\int_\Omega G(\tilde{\xi},\xi_3;x)\Big[\Delta p(\tilde{\xi},\xi_3;x_s)+(q^o)^2p(\tilde{\xi},\xi_3;x_s)\Big]d\tilde{\xi}d\xi_3\notag\\[2mm]
&=&\int_\Omega G(\tilde{\xi},\xi_3;x)\Big[\tilde{q}p(\tilde{\xi},\xi_3;x_s)-\delta(\tilde{\xi}-\tilde{x}_s)\delta(\xi_3-x_{s3})\Big]d\tilde{\xi}d\xi_3\notag\\[2mm]
&=&\int_\Omega G(\tilde{\xi},\xi_3;x)\tilde{q}p(\tilde{\xi},\xi_3;x_s)\,d\tilde{\xi}d\xi_3\,.\label{eq:G_equal}
\end{eqnarray}
Now we are readily derive
\begin{equation}\label{eq:uOme}
p(\tilde{x},x_3;x_s)=-\int_\Omega G(\tilde{\xi},\xi_3;x)\tilde{q}p(\tilde{\xi},\xi_3;x_s)d\tilde{\xi}d\xi_3+\int_{\partial\Omega}(p_+-p_-)\frac{\partial G}{\partial\nu}ds+G(\tilde{x}_s,x_{s3};x) \quad \text{for}\quad x\in\mathbb{R}_h^3\,,
\end{equation}
where $p_+$ and $p_-$ are the limits of $p(\tilde{\xi},\xi_3;x_s)$ as $(\tilde{\xi},\xi_3)$ approaches $\partial \Omega$ from the exterior and interior of $\Omega$; we refer to \cite{GX4} for more detail.
If we introduce 
\begin{equation}\label{eq:phiOri}
\phi(\tilde{\xi},\xi_3)=p_+(\tilde{\xi},\xi_3;x_s)-p_-(\tilde{\xi},\xi_3;x_s)\,, 
\end{equation}
then we can rewrite \eqref{eq:uOme} as
\begin{equation}\label{eq:uRepresen}
p(\tilde{x},x_3;x_s)+\int_\Omega G(\tilde{\xi},\xi_3;x)\tilde{q}p(\tilde{\xi},\xi_3;x_s)d\tilde{\xi}d\xi_3-\int_{\partial\Omega}\phi(\tilde{\xi},\xi_3)\frac{\partial G}{\partial\nu}(\tilde{\xi},\xi_3;x)ds=G(\tilde{x}_s,x_{s3};x) \quad \text{for}\quad x\in\Omega.
\end{equation}
From the condition \eqref{eq:u2u4Dir}, we can compute as follows:
\begin{eqnarray*}
0&=&\rho_2 p_+(x;x_s)-\rho_4p_-(x;x_s)\\[2mm]
&=&\rho_2\Big[\frac{1}{2}\phi(x)+\int_{\partial\Omega}\phi(\tilde{\xi},\xi_3)\frac{\partial G}{\partial\nu}(\tilde{\xi},\xi_3;x)ds\Big]-\rho_4\Big[-\frac{1}{2}\phi(x)+\int_{\partial\Omega}\phi(\tilde{\xi},\xi_3)\frac{\partial G}{\partial\nu}(\tilde{\xi},\xi_3;x)ds\Big]\\[2mm]
&&+(\rho_2-\rho_4)\Big[G(\tilde{x}_s,x_{s3};x)-\int_\Omega G(\tilde{\xi},\xi_3;x)\tilde{q}p(\tilde{\xi},\xi_3;x_s)d\tilde{\xi}d\xi_3\Big] \quad \text{for}\quad x\in\partial\Omega\,,
\end{eqnarray*}
which implies 
\begin{eqnarray}\label{eq:phi}
&&\phi(x)+\frac{2(\rho_2-\rho_4)}{\rho_2+\rho_4}\int_{\partial\Omega}\phi(\tilde{\xi},\xi_3)\frac{\partial G}{\partial\nu}(\tilde{\xi},\xi_3;x)ds-\frac{2(\rho_2-\rho_4)}{\rho_2+\rho_4}\int_\Omega G(\tilde{\xi},\xi_3;x)\tilde{q}p(\tilde{\xi},\xi_3)d\tilde{\xi}d\xi_3\notag\\[2mm]
&=&-\frac{2(\rho_2-\rho_4)}{\rho_2+\rho_4}G(x_s;x)\quad \text{for}\quad x\in\partial\Omega\,,
\end{eqnarray}
where we have applied the facts that
\[G(\tilde{x}_s,x_{s3};x)-\int_\Omega G(\tilde{\xi},\xi_3;x)\tilde{q}p(\tilde{\xi},\xi_3;x_s)d\tilde{\xi}d\xi_3\]
is continuous across $\partial\Omega$ and
\[v(\tilde{x},x_3)=\int_{\partial\Omega}\phi(\tilde{\xi},\xi_3)\frac{\partial G}{\partial\nu}(\tilde{\xi},\xi_3;x)ds\]
satisfies the following jump conditions when $x\in\partial\Omega$,
\begin{eqnarray*}
v^+(x)&=&\frac{1}{2}\phi(x)+\int_{\partial\Omega}\phi(\tilde{\xi},\xi_3)\frac{\partial G}{\partial\nu}(\tilde{\xi},\xi_3;x)d\tilde{\xi}d\xi_3,\\[2mm]
v^+(x)&=&-\frac{1}{2}\phi(x)+\int_{\partial\Omega}\phi(\tilde{\xi},\xi_3)\frac{\partial G}{\partial\nu}(\tilde{\xi},\xi_3;x)d\tilde{\xi}d\xi_3,\\[2mm]\phi(x)&=&v^+(x)-v^-(x).
\end{eqnarray*}

For the system of integral equations \eqref{eq:uRepresen} and \eqref{eq:phi}, we have the following result. 
\begin{Theorem}
If $\max_{x\in \Omega} |\tilde{q}|$ and $|\rho_2-\rho_4|$ are sufficiently small, then the system of integral equations \eqref{eq:uRepresen} and \eqref{eq:phi} has a unique solution.
\end{Theorem}
\begin{proof} We first represent \eqref{eq:uRepresen} into the following form
\begin{equation}\label{eq:operator_u}
p+\mathbf{Q}p=f,
\end{equation}
where $Q$ and $f$ are given by 
\[\mathbf{Q}p(x;x_s)=\int_\Omega G(\tilde{\xi},\xi_3;x)\tilde{q}p(\tilde{\xi},\xi_3;x_s)d\tilde{\xi}d\xi_3\]
and
\[f(x)=\int_{\partial\Omega}\phi(\tilde{\xi},\xi_3)\frac{\partial G}{\partial\nu}(\tilde{\xi},\xi_3;x)ds+G(\tilde{x}_s,x_{s3};x).\]
If $\max_{x\in \Omega} |\tilde{q}|$ is sufficiently small, then the operator $\mathbf{I+Q}$ has a bounded inverse 
$(\mathbf{I+Q})^{-1}$, where $\mathbf{I}$ denotes the identical operator in $L^2(\Omega)$. We can obtain the following equation by substituting $u=(\mathbf{I+Q})^{-1}f$ into \eqref{eq:phi}, 
\begin{equation}\label{eq:phi1}
\phi(x)+\frac{2(\rho_2-\rho_4)}{\rho_2+\rho_4}\mathbf{P}\phi(x)-\frac{2(\rho_2-\rho_4)}{\rho_2+\rho_4}\mathbf{Q}\circ(\mathbf{I+Q})^{-1}f(x)=-\frac{2(\rho_2-\rho_4)}{\rho_2+\rho_4}G(x_s;x),
\end{equation}
where $\mathbf{P}$ is given by 
\[\mathbf{P}\phi(x)=\int_{\partial\Omega}\phi(\xi)\frac{\partial G}{\partial\nu}(\xi;x)ds.\]
The operators $\mathbf{P}$ and $\mathbf{Q}\circ(\mathbf{I+Q})^{-1}$ are respectively bounded in $L^2(\partial\Omega)$ and $L^2(\Omega)$. Consequently, \eqref{eq:phi1} has a unique solution $\phi\in L^2(\partial\Omega)$ when $|\rho_2-\rho_4|$ is small enough. Substituting this $\phi$ into \eqref{eq:operator_u}, we then obtain the unique solution $u$.
\end{proof}

For the case that $\rho_2=\rho_4$, we can easily see from \eqref{eq:phiOri} that 
\[\phi(\tilde{\xi},\xi_3)=0,\quad (\tilde{\xi},\xi_3)\in\partial\Omega\,.\]
The system of the integral equations \eqref{eq:uRepresen} and \eqref{eq:phi} 
reduce to a single integral equation
\begin{equation}\label{eq:uSimply}
p(x;x_s)=-\int_\Omega G(\tilde{\xi},\xi_3;x)\tilde{q}p(\tilde{\xi},\xi_3;x_s)d\tilde{\xi}d\xi_3+G(x;x_s) \quad \text{for}\quad x\in\Omega.
\end{equation}

For our subsequent numerical reconstruction of the point source, we need to compute the observation data. 
For this purpose, 
we propose an iterative method to solve the integral equation \eqref{eq:uSimply}. 
Choosing the initial guess $p^{(1)}(\tilde{\xi},\xi_3;x_s)=G(\tilde{\xi},\xi_3;x_s)$ for $(\tilde{\xi},\xi_3)\in\Omega$, then
for $n\ge 1$ we can generate a sequence of the approximations to $p(x)$ by 
\begin{equation}\label{eq:uiter}
p^{(n+1)}(x;x_s)=G(x;x_s)-\int_\Omega G(\tilde{\xi},\xi_3;x)\tilde{q}\,p^{(n)}(\tilde{\xi},\xi_3;x_s)d\tilde{\xi}d\xi_3\,,
\quad x\in\Omega\,.
\end{equation}

Once the total field $p(x)$ or its approximation is available for $x\in \Omega$, the scattered field $p^s$ can be calculated by
\begin{equation}\label{eq:u_scattered}
p^s(x^r;x_s)=k^2\int_\Omega G(x;x^r)\tilde{q}\,p(x;x_s)dx,\quad x^r\in\Gamma_r,
\end{equation}
where $\Gamma_r$ is the location of the receivers, and $x_s$ is the source location.

\section{A multilevel sampling method for locating an unknown acoustic source
} \label{MSM}

As we recall, the inverse problem of our interest is to determine the location of a time-harmonic sound source in a three-layered stratified waveguide. In this section, we present a multilevel sampling method for solving the inverse source problem.

Locating acoustic sources in waveguides has been widely studied, 
say, \cite{ba, bu, pd, sh1, sh2}. One of the
significant methods is the ``matched-field processing''
method which was proposed in 1976 in \cite{bu}.  
In the early 1990's, the matched-field method was combined with 
the scattering theory in a shallow ocean (ref. \cite{GX1, GX2, xp, XU1} ) to locate an 
acoustic source in a two-dimensional waveguide, see \cite{XY1, XY2}.

However, there remains a concern when we apply the method to locate 
an acoustic source in a three-dimensional waveguide 
in the inverse scattering theory of shallow ocean, particularly from the computational 
aspect. That is, the separation of source and sampling points
in the scattering formula is no longer clear due to an inhomogeneous inclusion in the waveguide.
This leads to a large computational burden. 

Next, we shall present a new matched-field indicator and an effective multilevel numerical algorithm to locate the unknown sound source.
We first define a matched-field indicator. 
Let $\{p^s_m\}$ be the detected data set consisting of the scattered field $p^s_{m}(x_m^r;x_s)$ sampled 
at each receiver situated at $x^r_m$, where $m=1,2,\cdots,M$. We construct the following index function:
\begin{equation}\label{eq:index1}
I(x):=\bigg\{\sum_{m=1}^M \Big|p^s(x_m^r;x)-p^s_{m}(x_m^r;x_s)\Big|^2\bigg\}^{-1}\quad \text{for} \quad x\in D,
\end{equation}
where $p^s(x^r_m;x)$ is the computed scattered field of each sampling source point $x$ and computed by
\eqref{eq:uiter} and \eqref{eq:u_scattered}, and $D$ is the sampling region. 
In practice, we normalize the above index function as
\begin{equation}\label{eq:index}
I(x):=\frac{I(x)}{\displaystyle\max_{x\in D} I(x)},\quad  x\in D.
\end{equation}

It is easy to note that the index function $I(x)$ close to 1 when the sampling point $x$ is near to the location of 
the acoustic source, otherwise it is small. Hence we can use $I(x)$ to help locate the source point. 
However, we need to evaluate $I(x)$ for every sampling point, or compute it for a total number of $\mathcal{O}(m^3)$ times 
when we use a sampling domain with an $m\times m\times m$ mesh. 
This can be still rather expensive computationally for a very large $m$. 
We are now going to formulate a multilevel sampling algorithm, which can essentially reduce 
the computational effort in evaluating the index function $I(x)$. 

The new multilevel algorithm is motivated by a simple observation. 
We can easily see that the index function $I(x)$ vanishes in the entire sampling domain except 
at the source point, so we just need to focus on those sampling points at which $I(x)$ 
is relative large, repeating the procedure iteratively with an initial relatively coarse sampling domain. 
This leads us to the following multilevel sampling algorithm.

\ms 
{\bf Multilevel sampling algorithm}
\begin{enumerate}
\item Select a sampling domain $D$ that contains the unknown source point $x_s$.\\[2mm]
Choose a uniform (coarse) mesh on $D$, consisting of cubic elements of equally size, denoted 
by $D_0$.\\[2mm]
Select a cut-off value $c$, a number $N$ for the maximum iterations,  and set $n:=1$.
\item Compute the index function value $I(x)$ for each sampling 
point
$x\in D_{n-1}$ using \eqref{eq:index}.
If $I(x)\geq c$ at a sampling point $x$, select all the vertices of the cubic elements in $D_{n-1}$ 
that share $x$ as one of their vertices, otherwise drop the grid point. 
Update $D_{n-1}$ by all those selected grid points.
\item If $n\leq N$, refine the mesh $D_{n-1}$ to $D_n$, then set $n:=n+1$ and go to step 2; otherwise, set $D_n:=D_{n-1}$ and go to step 4.
\item Output all the grid points in $D_n$ for the approximate position of the point source $x_s$.
\end{enumerate}

\begin{Remark}
One may apply any existing refinement strategy for the refinement required in step 3 of the 
multilevel sampling algorithm. 
In all our numerical simulations reported in the next section, 
we have adopted the simple bisection technique, namely 
each cubic element is partitioned into 8 equal smaller cubes.
\end{Remark}

We can easily see that the multilevel sampling algorithm involves only matrix-vector operations, 
no any optimization process, matrix inversions or the solution of large-scale ill-posed linear systems.
Moreover, its major cost is to update the indicator function by the explicit formula \eqref{eq:index} at each iteration, and the computational sampling region $D_n$ shrinks iteratively based on the cut-off value $c$. 
Compared with most existing sampling methods of computational complexity $\mathcal{O}(m^3)$ 
($m$ is the number of sampling points in each direction), 
this multilevel algorithm is much less expensive. 
In addition, the cut-off value $c$ is relatively easy to select and insensitive to the noise level. 
In practice, one may simply take the cut-off value to be 0.9 or 0.95 in each iteration. 

Consequently, we can conclude that our novel algorithm is a direct sampling method, and is extremely simple and inexpensive. 
We may observe in the following numerical section that the algorithm can work well with partial scattered data as well as 
with a rather small number of receivers. In terms of these aspects, this novel multilevel sampling algorithm outperforms some popular existing sampling methods such as the well-known linear sampling type methods \cite{BL, LXZ1, XML}, where the cut-off values are very sensitive to the noise and difficult to choose, and the number of incidences should be large. It is worth mentioning that our multilevel sampling algorithm to be presented here is essentially different in nature from the multilevel linear sampling developed in \cite{LZ}: the cut-off values $c$ is much easier to select and fixed during the iterations, and it works with much less receivers. Moreover, the proposed multilevel algorithm is robust against noise in the observed data.

\section{Numerical simulations} \label{NS}
In this section, we provide several numerical experiments to verify the effectiveness and efficiency of the newly proposed multilevel sampling method. All the programs in our experiments are written in MATLAB and run on a 2.83 GHz PC with 16GB memory.

We first list the parameters that are employed in our numerical examples. The depth of the ocean waveguide 
$h=100$, the heights of 
the first and second layers $\Gamma_1$ and $\Gamma_2$ of the waveguide 
are $d_1=100/3$ and $d_2=200/3$ respectively. 
The wavenumber $k_i=2\pi f/c_i$ $(i=1,2,3)$, where $f=75$, $c_1=1000$, $c_2=1500$ and $c_3=3000$. The refractive index $n_1=1$, $n_2=1/2$ and $n_3=1/3$. The densities $\rho_1=1000$, $\rho_2=1500$, and $\rho_3=3000$. 

We implement the iterative method \eqref{eq:uiter} to generate the total field $p$ and terminate the iteration when the relative $L^2$-norm error
between two fields of consecutive iterations is less than $\varepsilon=10^{-3}$. Then the scattered filed $p^s$ is obtained by the formula \eqref{eq:u_scattered}. Moreover, random noises are added to the exact scattered field in the following form:
\[p^s_\delta(x^r):=p^s(x^r)\Big[1+\delta\Big(r_1(x^r)+ir_2(x^r)\Big)\Big], \quad x^r\in\Gamma_r,\]
where $r_1(x^r)$, $r_2(x^r)$ are two normal random numbers
varying from -1 to 1, and $\delta$ represents the level of noise which is usually taken as $10\%$ unless specified otherwise. 
The mesh size of the forward problem is selected as 
 $1/15$, while the initial mesh size for the inverse problem is
$4$. The cut-off value $c$ in the multilevel sampling algorithm is chosen as 0.95 in the following numerical tests. In the first three numerical examples, the sampling region $D$ is set to be $[10,40]\times[10,40]\times[10,40]$. We select the sampling region $D$ as $[10,40]\times[10,40]\times[25,55]$ in the last numerical test. 

{\bf Example 1.} We consider a point source $x_s$ situated at $(18,18,25)$, see the red point in figure \ref{fig:ex11}(a). Only 5 receivers are employed, at the locations $\Gamma_{r_{11}}=(70, 10+5n, 90)$, $n=0,1,\cdots,4$, see the black points of $\Gamma_{r_{11}}$ in \ref{fig:ex11}(a). 
The known scatterer is assumed to be in the second layer, but not contained in the sampling domain $D$; 
see the cyan cube located at the region $[32,34]\times[32,34]\times[42,44]$ in figure \ref{fig:ex11}(a).
%

\begin{figurehere}
 \hfill{}\includegraphics[clip,width=0.45\textwidth]{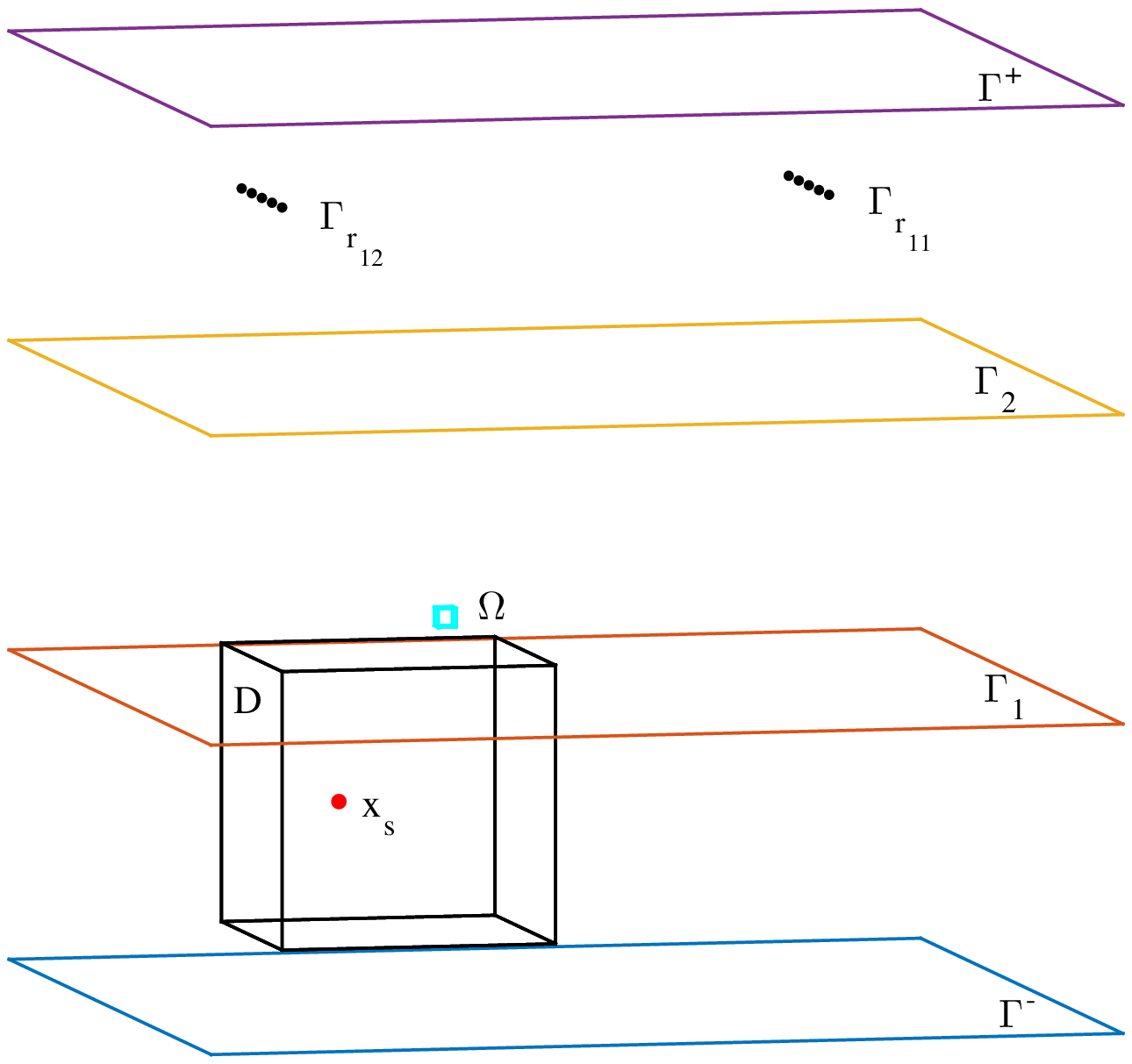}\hfill{}
 \hfill{}\includegraphics[clip,width=0.45\textwidth]{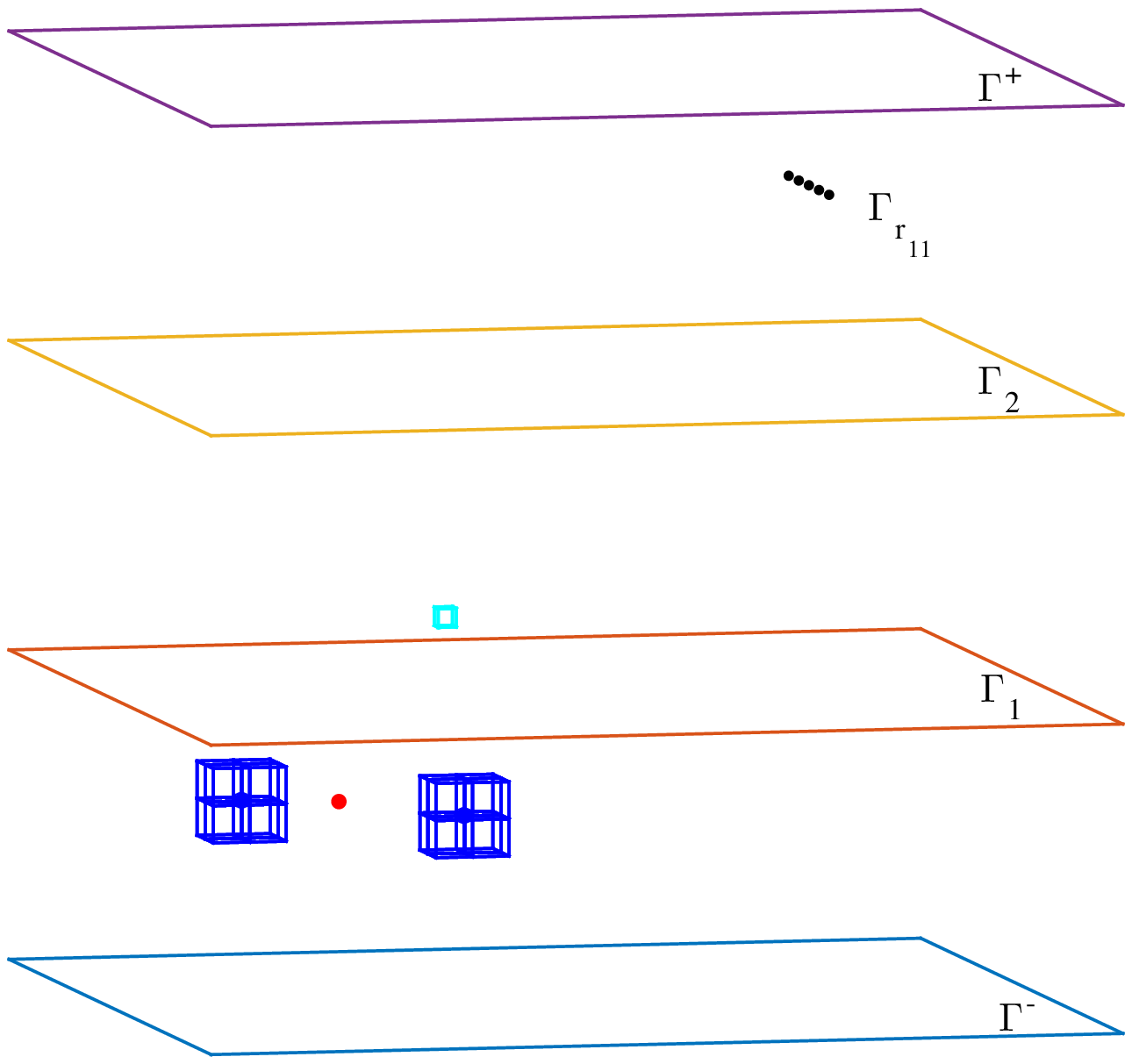}\hfill{}
 
 \hfill{}(a)\hfill{} \hfill{}(b)\hfill{}
 
 \caption{\label{fig:ex11} \small{(a) Location of the true point source $x_s$ in Example 1; (b) the reconstruction at the 
 1st iteration using the receivers on $\Gamma_{r_{11}}$.}}
 \end{figurehere}
 The numerical reconstruction of the first iteration is shown in Figure \ref{fig:ex11}(b).  
Because of the strong refraction of the scattered wave between each layer, the receivers in $\Gamma_{r_{11}}$ can obtain almost no scattered data in the layer $M_3$\,. Therefore, the blue cubes in Figure \ref{fig:ex11}(b) provide 
a poor location of the source. This numerical result explains that the source contained in the bottom layer may not be detected when the receivers are placed near the surface of the ocean.
  
As the vertical scattered waves can propagate without refraction, we then put the receivers in $\Gamma_{r_{12}}=(10, 10+5n, 90)$, $n=0,1,\cdots,4$, see the black points of $\Gamma_{r_{12}}$ in \ref{fig:ex11}(a). The reconstructed results are provided in Figure \ref{fig:ex12}.

\begin{figurehere}
 \hfill{}\includegraphics[clip,width=0.31\textwidth]{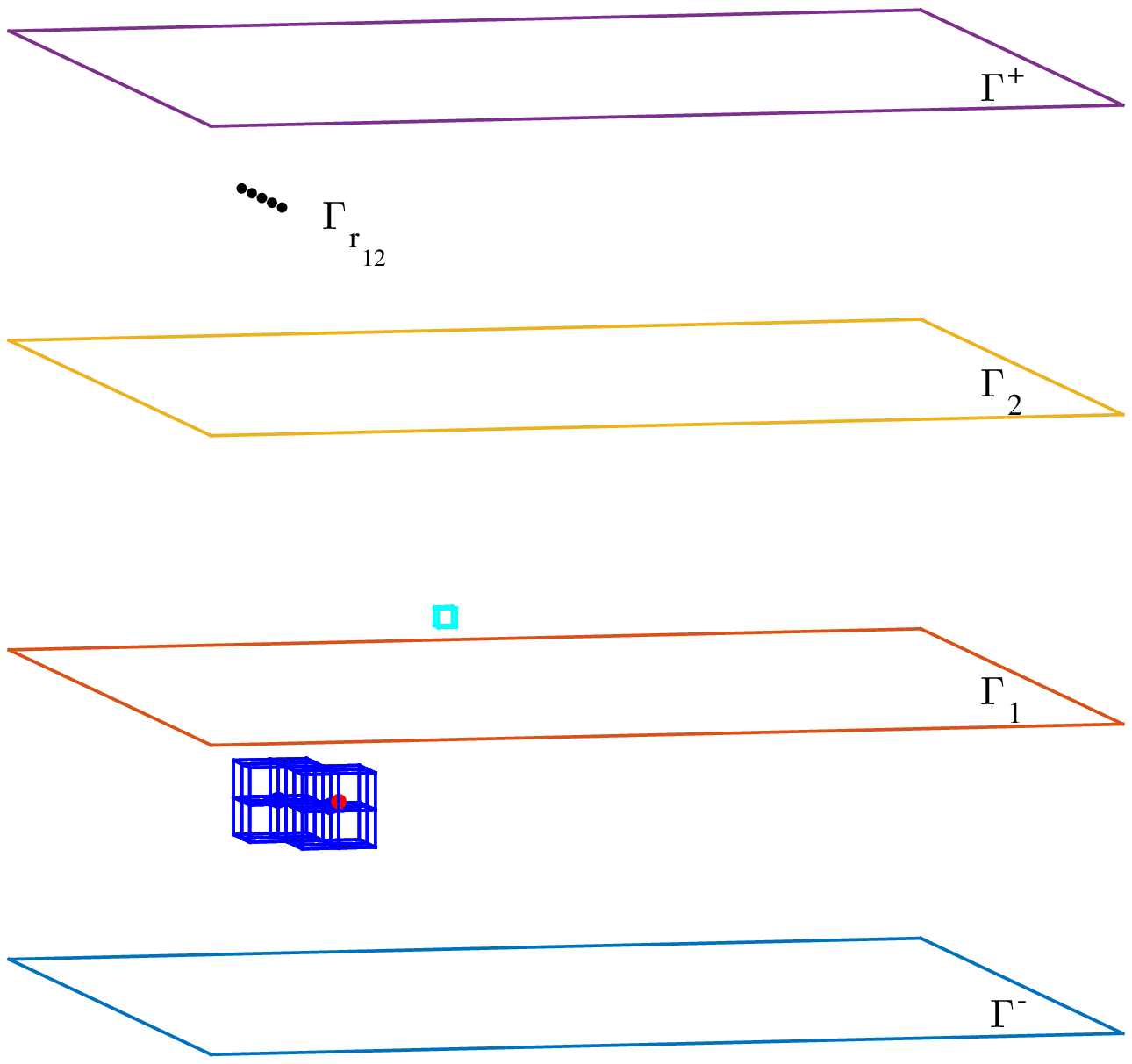}\hfill{}
 \hfill{}\includegraphics[clip,width=0.31\textwidth]{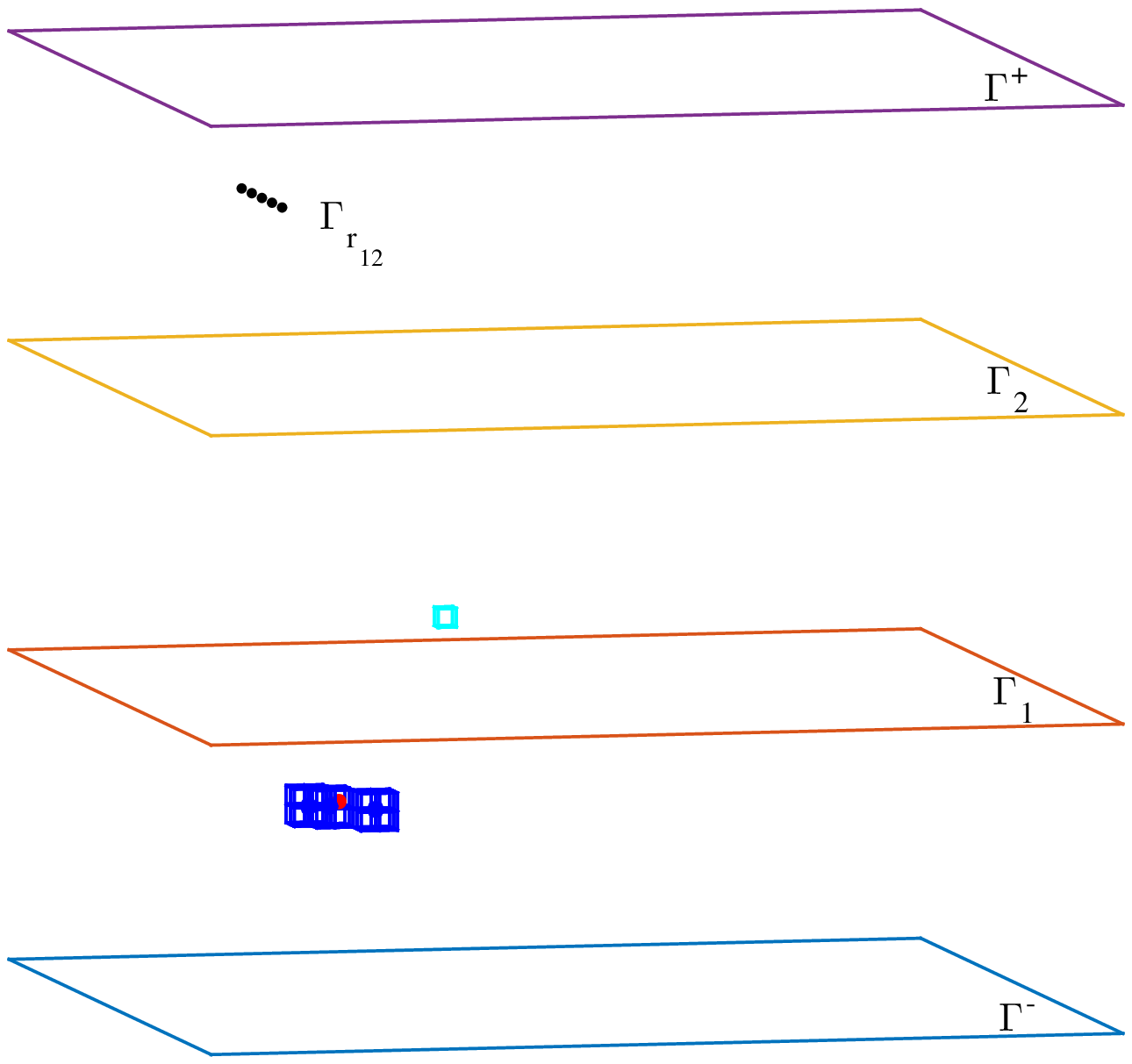}\hfill{}
 \hfill{}\includegraphics[clip,width=0.31\textwidth]{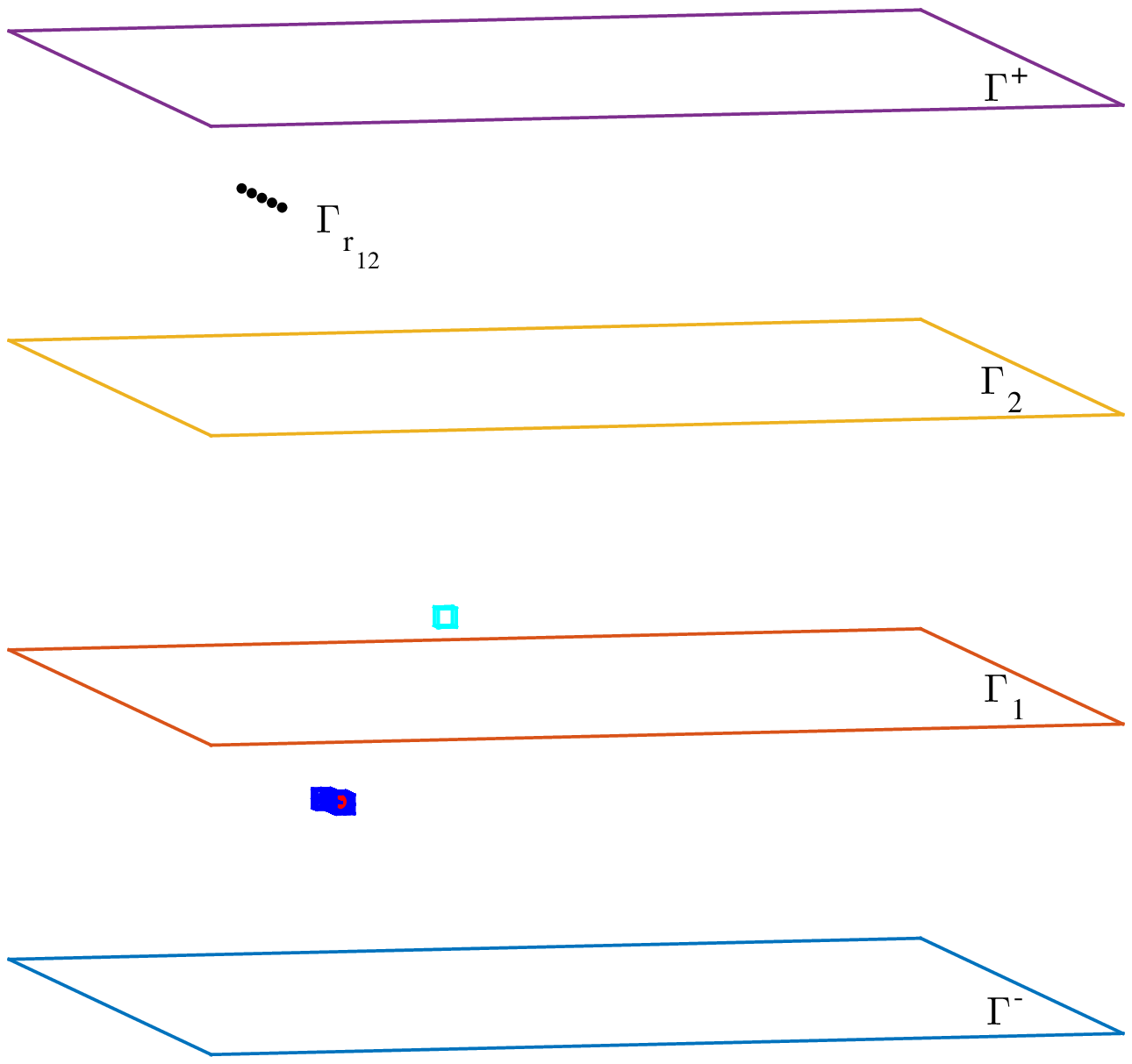}\hfill{}
 
 \hfill{}(a)\hfill{} \hfill{}(b)\hfill{} \hfill{}(c)\hfill{}
 
 \caption{\label{fig:ex12} \small{(a)-(c) the reconstructions by the 1st to 3rd iterations in Example 1 
 using the receivers on $\Gamma_{r_{12}}$.}}
 \end{figurehere}

We can observe that the reconstruction of the source in Figure \ref{fig:ex12}(c) is much better compared with the one in Figure \ref{fig:ex11}(b). Accordingly, we can conclude that the newly proposed method can provide 
a rather reliable estimation of the position of the point source with the vertical partial scattered data, and 
the locations of receivers are very important for detecting the point source $x_s$.

{\bf Example 2.} This example considers the same point source $x_s$ and known scatterer $\Omega$ 
as in Example 1, but 5 other receivers are applied, at the positions $\Gamma_{r_{21}}=(60, 60+5n, 60)$, $n=0,1,\cdots,4$;
see the black points of $\Gamma_{r_{21}}$ in Figure \ref{fig:ex2_demo}.

\begin{figurehere}
 \hfill{}\includegraphics[clip,width=0.45\textwidth]{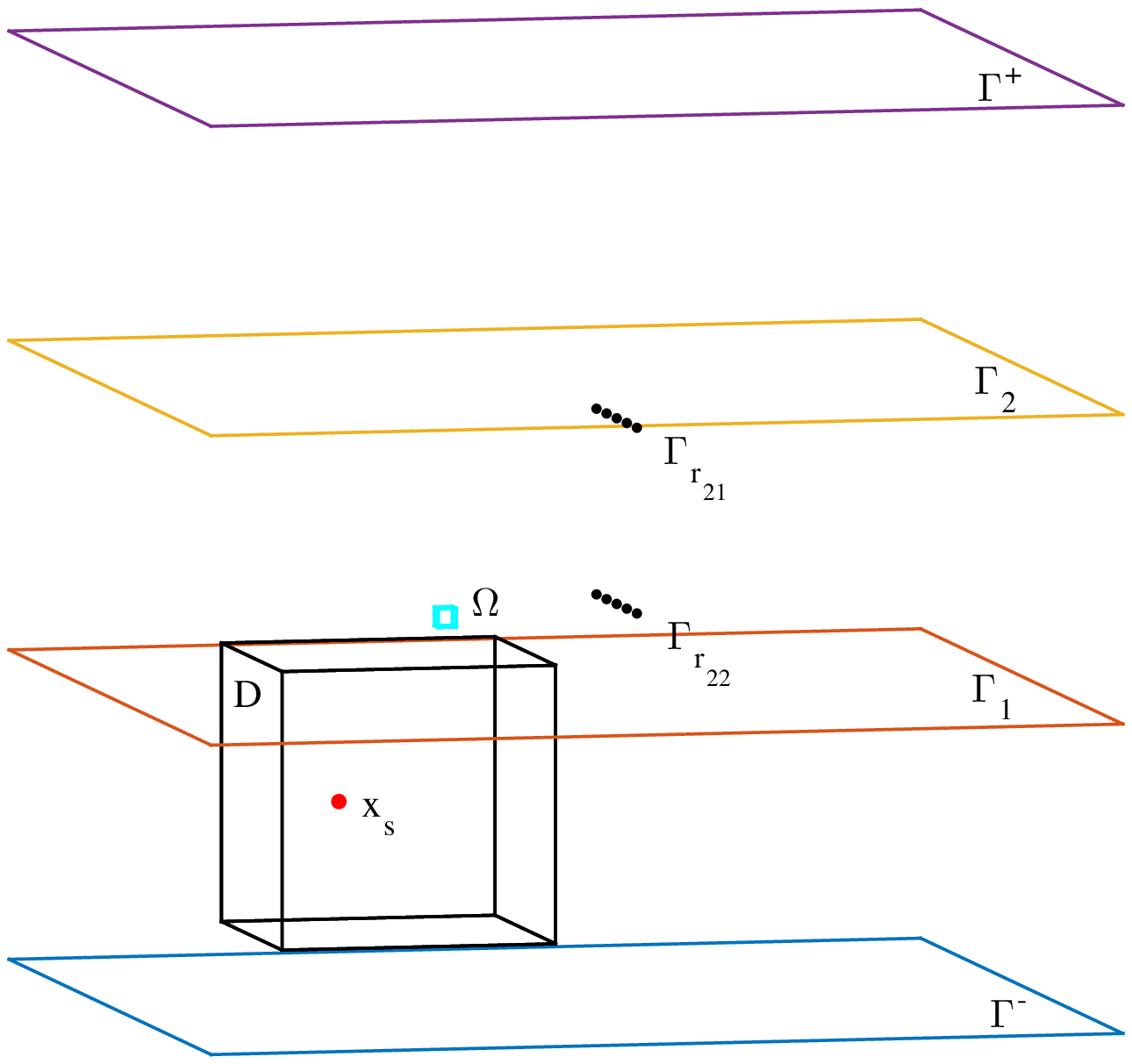}\hfill{}
 \vskip -0.6 truecm
 \caption{\label{fig:ex2_demo} \small{The location of the true point source $x_s$ in Example 2.}}
 \end{figurehere}

 The reconstructions are presented in Figures \ref{fig:ex2}(a)-(c). 
 We can easily find that the receivers of $\Gamma_{r_{21}}$ are close to the surface $\Gamma_2$, 
 so only few refractive scattered waves can be received by $\Gamma_{r_{21}}$. 
 Although some additional cubes are seen in the reconstructions, the source $x_s$ is still located. 
 Then, we set the detecting devices close to $\Gamma_1$, namely, at $\Gamma_{r_{22}}=(60, 60+5n, 40)$, $n=0,1,\cdots,4$; see the black points of $\Gamma_{r_{22}}$ in Figure \ref{fig:ex2_demo}, and 
 the reconstructions are shown in Figures \ref{fig:ex2}(d)-(f).
 As the receivers in $\Gamma_{r_{22}}$ is nearer to layer $M_1$ than $\Gamma_{r_{21}}$, 
 so the scattered data can be measured much more in $\Gamma_{r_{22}}$ than in $\Gamma_{r_{21}}$. 
 Consequently, the recovered position of the source $x_s$ by $\Gamma_{r_{22}}$ is much better than that by $\Gamma_{r_{21}}$. 
 This example shows again that the positions of receivers are quite significant in recovering 
 the point source $x_s$.

\begin{figurehere}
 \hfill{}\includegraphics[clip,width=0.31\textwidth]{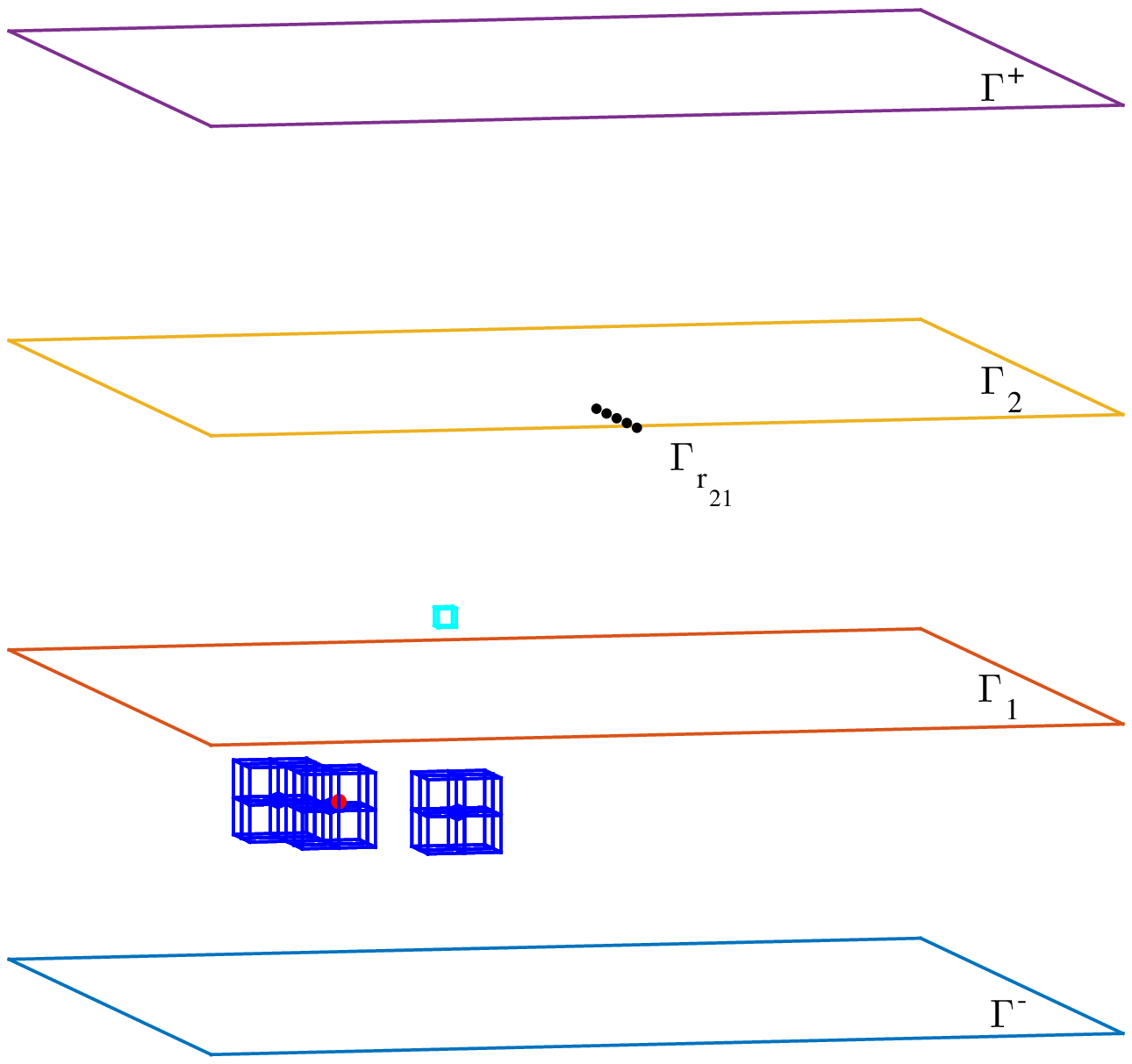}\hfill{}
 \hfill{}\includegraphics[clip,width=0.31\textwidth]{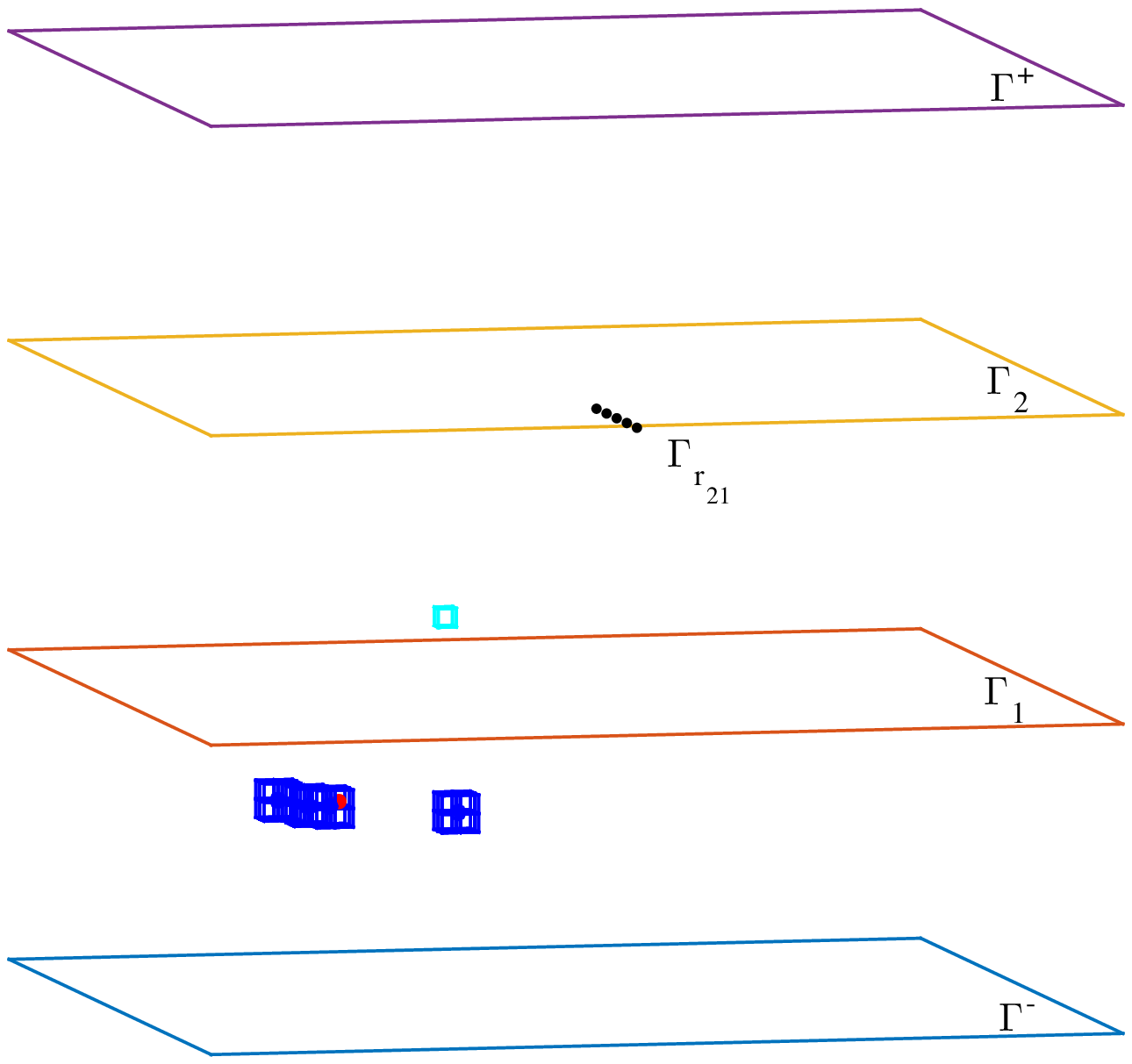}\hfill{}
 \hfill{}\includegraphics[clip,width=0.31\textwidth]{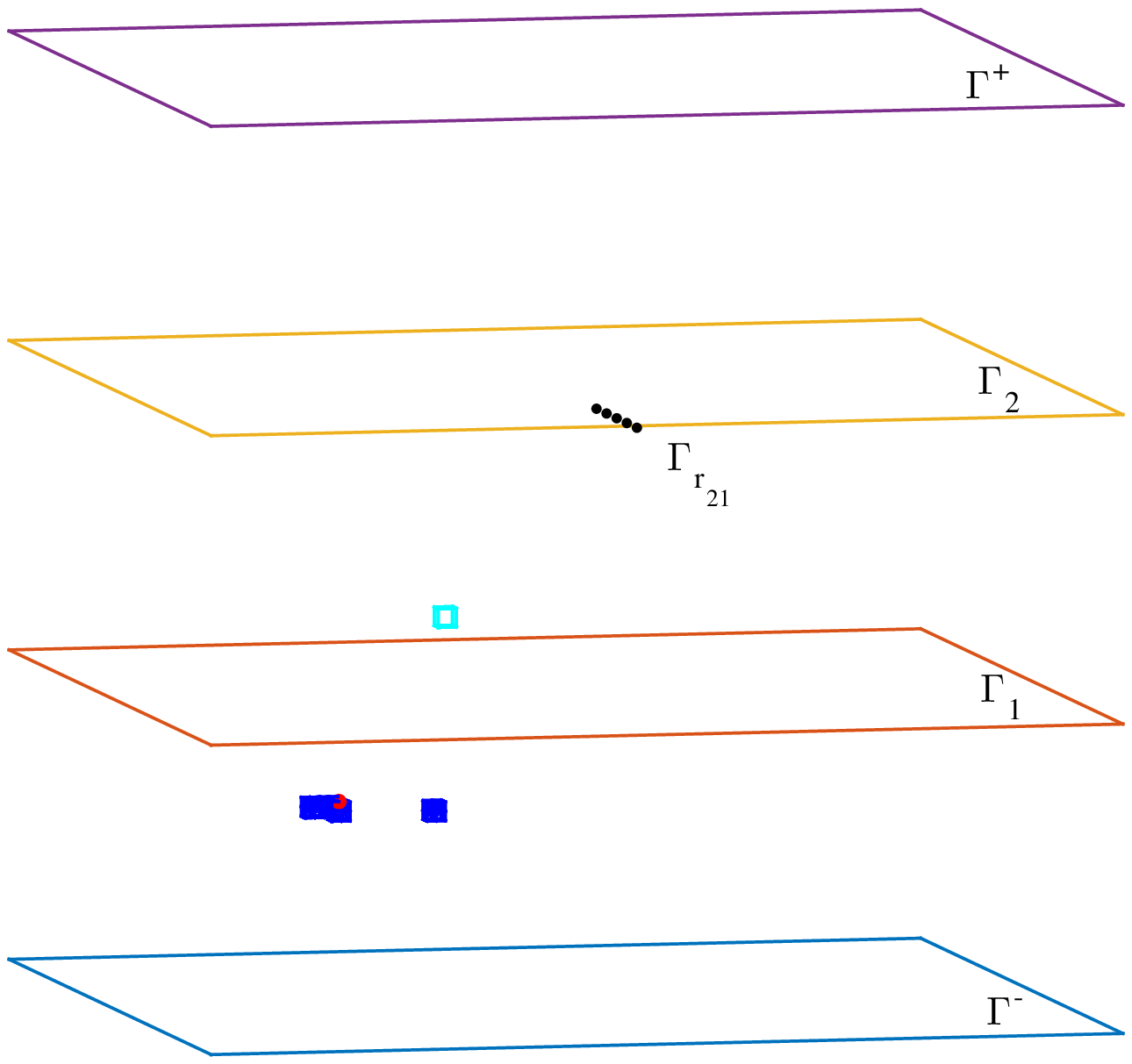}\hfill{}
 
 \hfill{}(a)\hfill{} \hfill{}(b)\hfill{} \hfill{}(c)\hfill{}
 
  \hfill{}\includegraphics[clip,width=0.31\textwidth]{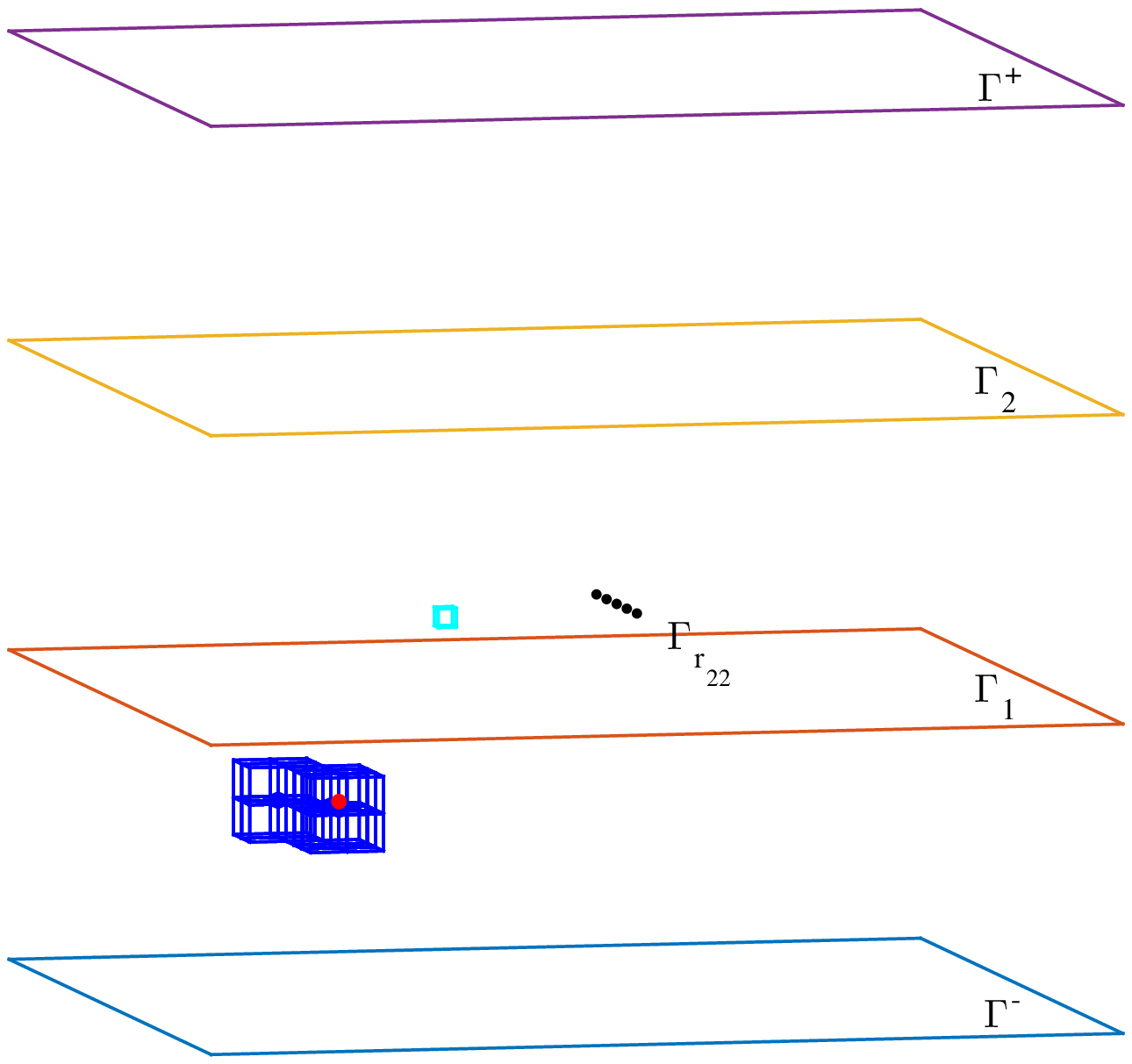}\hfill{}
 \hfill{}\includegraphics[clip,width=0.31\textwidth]{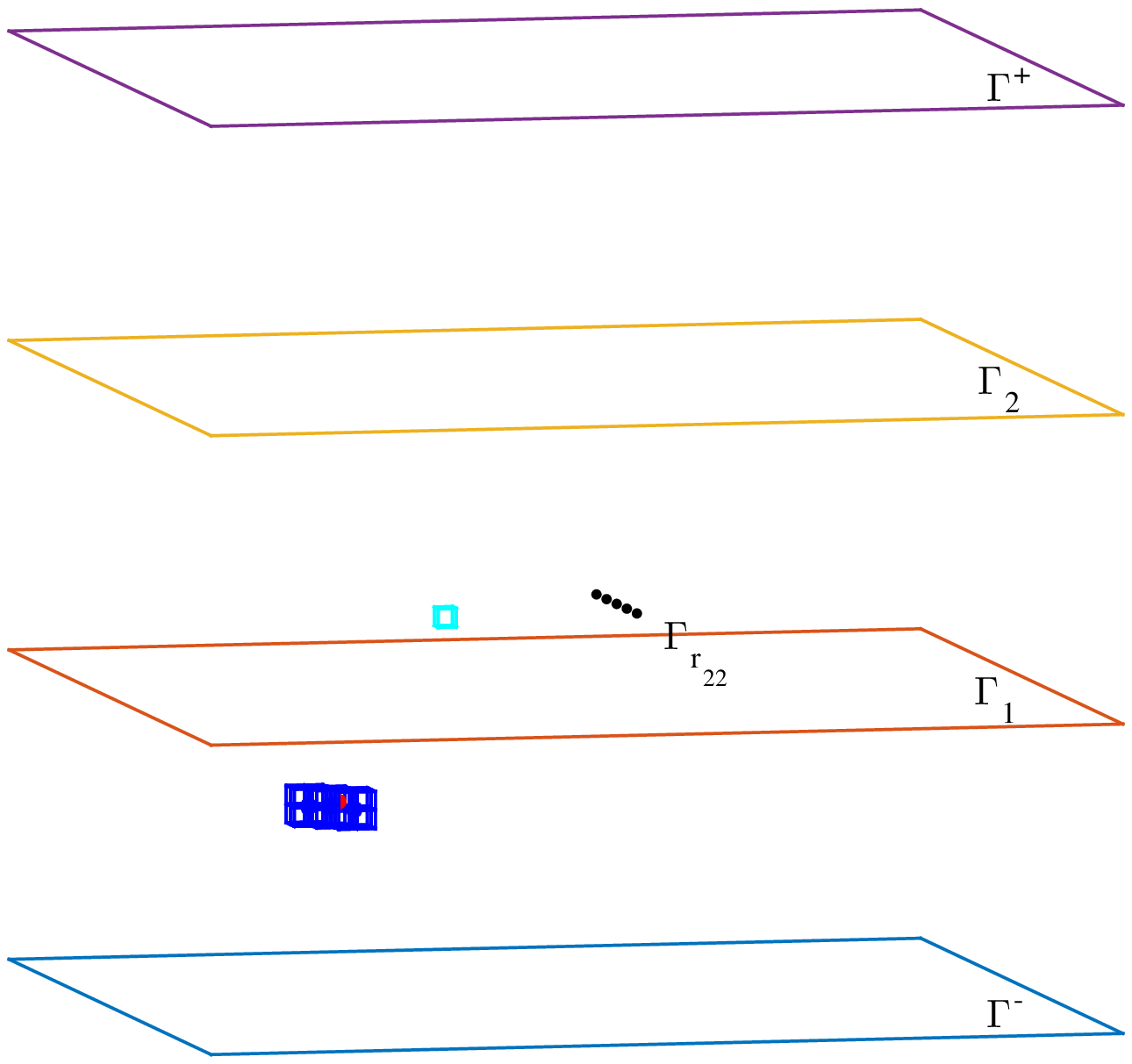}\hfill{}
 \hfill{}\includegraphics[clip,width=0.31\textwidth]{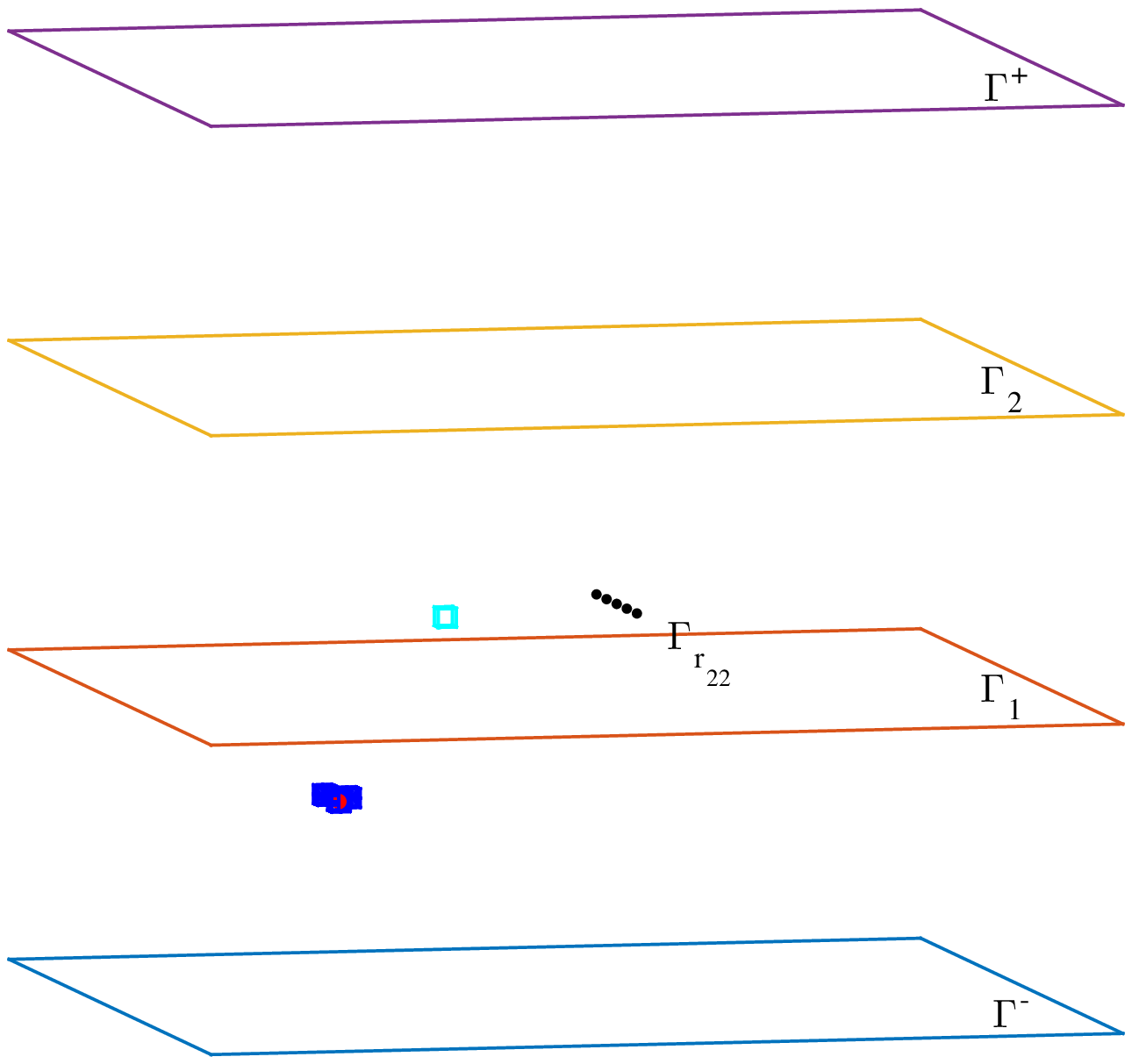}\hfill{}
 
 \hfill{}(d)\hfill{} \hfill{}(e)\hfill{} \hfill{}(f)\hfill{}
 
 \caption{\label{fig:ex2} \small{(a)-(c) (resp. (d)-(f)) the reconstructions by the 1st to 3rd iterations in Example 2 
 using the receivers on $\Gamma_{r_{21}}$ (resp. $\Gamma_{r_{22}}$).}}
 \end{figurehere}

{\bf Example 3.} In this experiment, This example considers the same point source $x_s$ and known scatterer $\Omega$ 
as in Example 1, but using another set of 5 receivers at the locations $\Gamma_{r_3}=(60, 60+5n, 30)$, $n=0,1,\cdots,4$;
see the black points of $\Gamma_{r_3}$ in Figure \ref{fig:ex3}(a). 

The reconstructions are shown in the Figure \ref{fig:ex3}. 
Without the refraction of scattered waves in the layer $M_1$, the position of the source is accurately found 
within just three iterations. In this case the new multilevel sampling method 
is quite  effective to provide a reliable estimation of the location of the source with a very small 
number of receivers.

\begin{figurehere}
 \hfill{}\includegraphics[clip,width=0.31\textwidth]{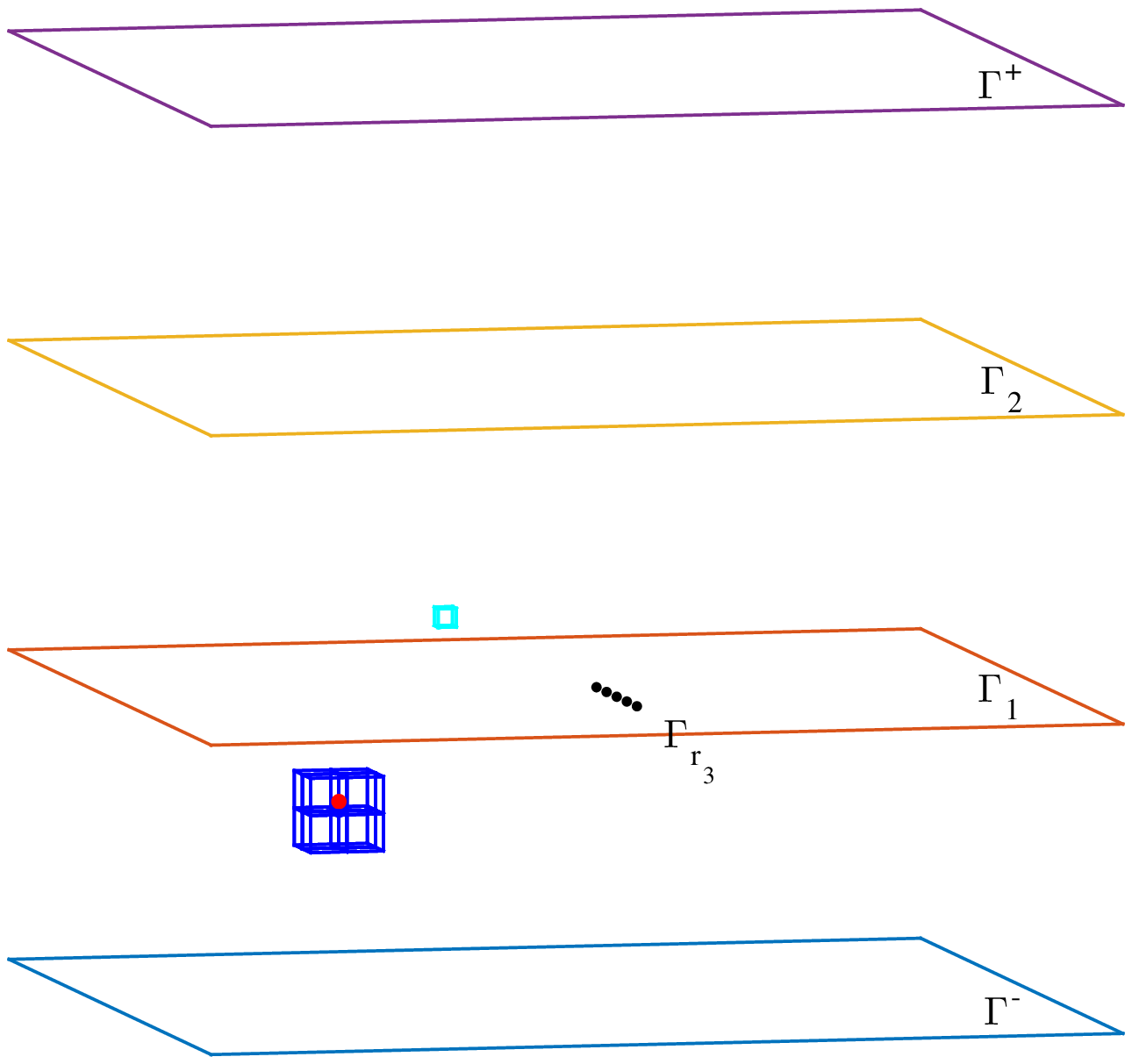}\hfill{}
 \hfill{}\includegraphics[clip,width=0.31\textwidth]{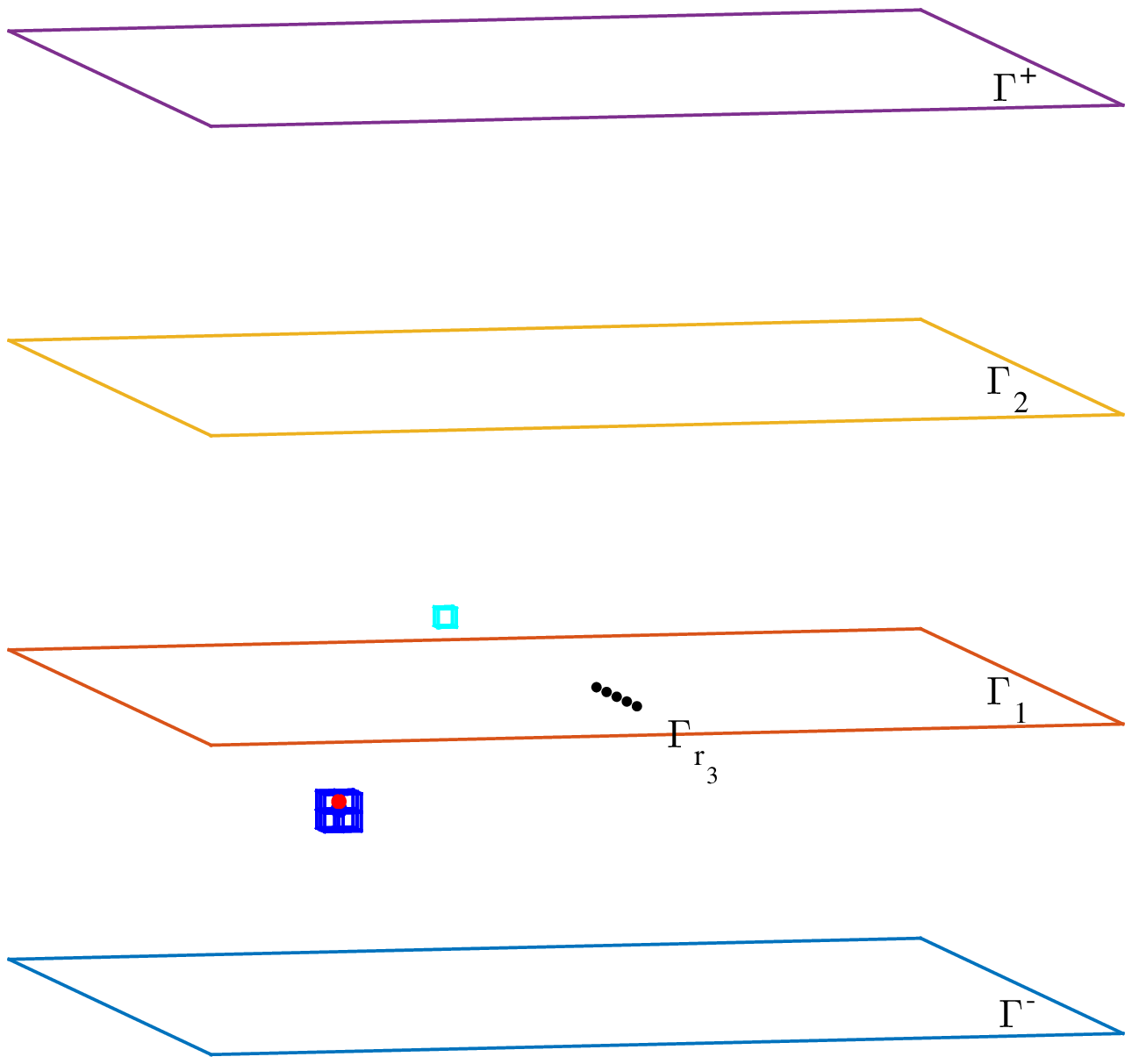}\hfill{}
 \hfill{}\includegraphics[clip,width=0.31\textwidth]{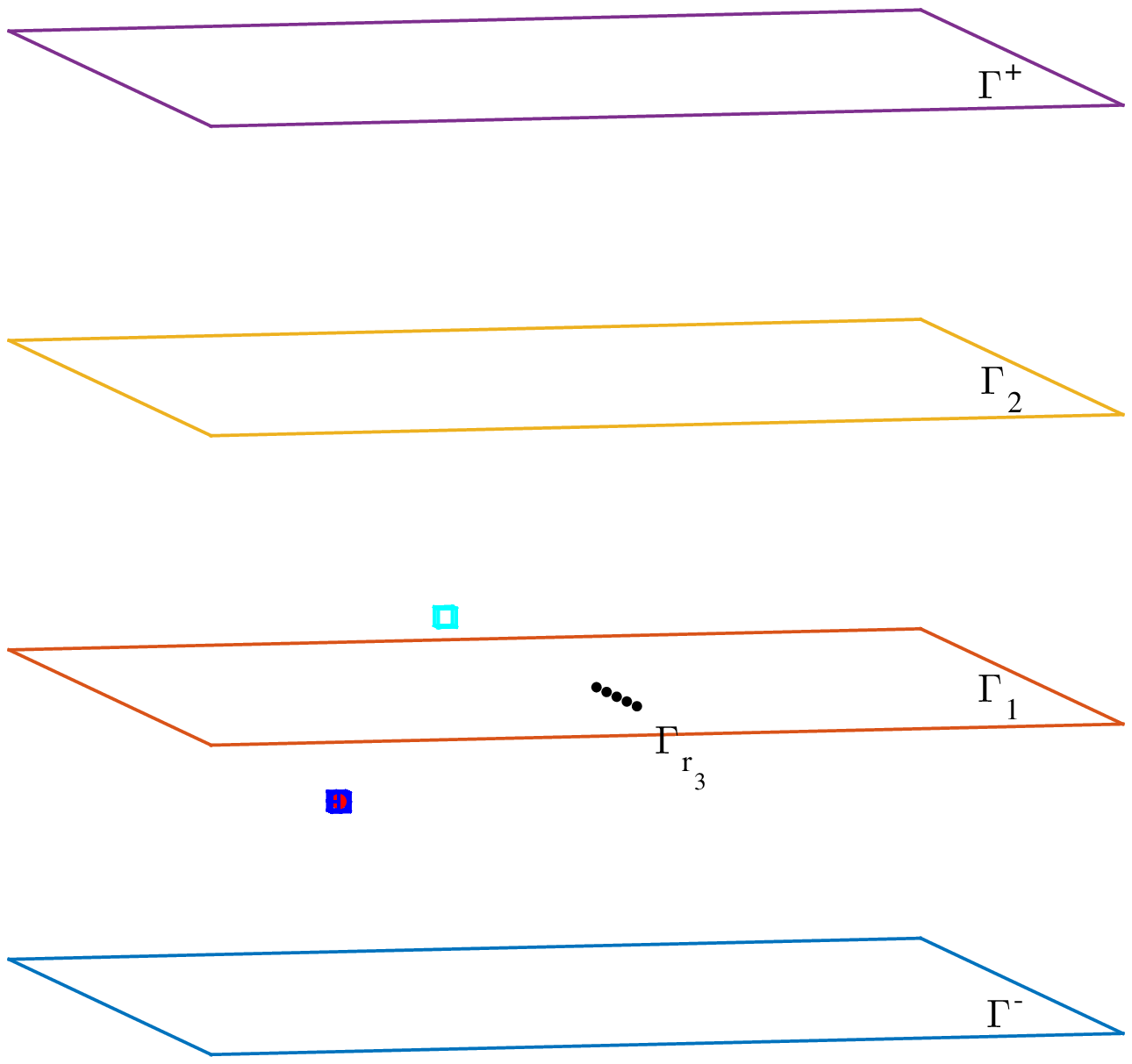}\hfill{}
 
 \hfill{}(a)\hfill{} \hfill{}(b)\hfill{} \hfill{}(c)\hfill{}
 
 \caption{\label{fig:ex3} \small{(a)-(c) the reconstructions by the 1st to 3rd iterations in Example 3 
 using the receivers on $\Gamma_{r_3}$.}}
 \end{figurehere}

{\bf Example 4.} This test investigates a point source $x_s$ locate at $(18,18,45)$;
see the red point in Figure \ref{fig:ex4_demo}.  Two sets of 5 receivers are used, at the positions $\Gamma_{r_{41}}=(60, 60+5n, 90)$ and $\Gamma_{r_{42}}=(60, 60+5n, 80)$, $n=0,1,\cdots,4$;
see the black points in Figure \ref{fig:ex4_demo}. The known scatterer $\Omega$ 
is set at $[46,48]\times[32,34]\times[42,44]$, which is outside the sampling region $D$;
see the cyan cube in Figure \ref{fig:ex4_demo}.  

\begin{figurehere}
 \hfill{}\includegraphics[clip,width=0.45\textwidth]{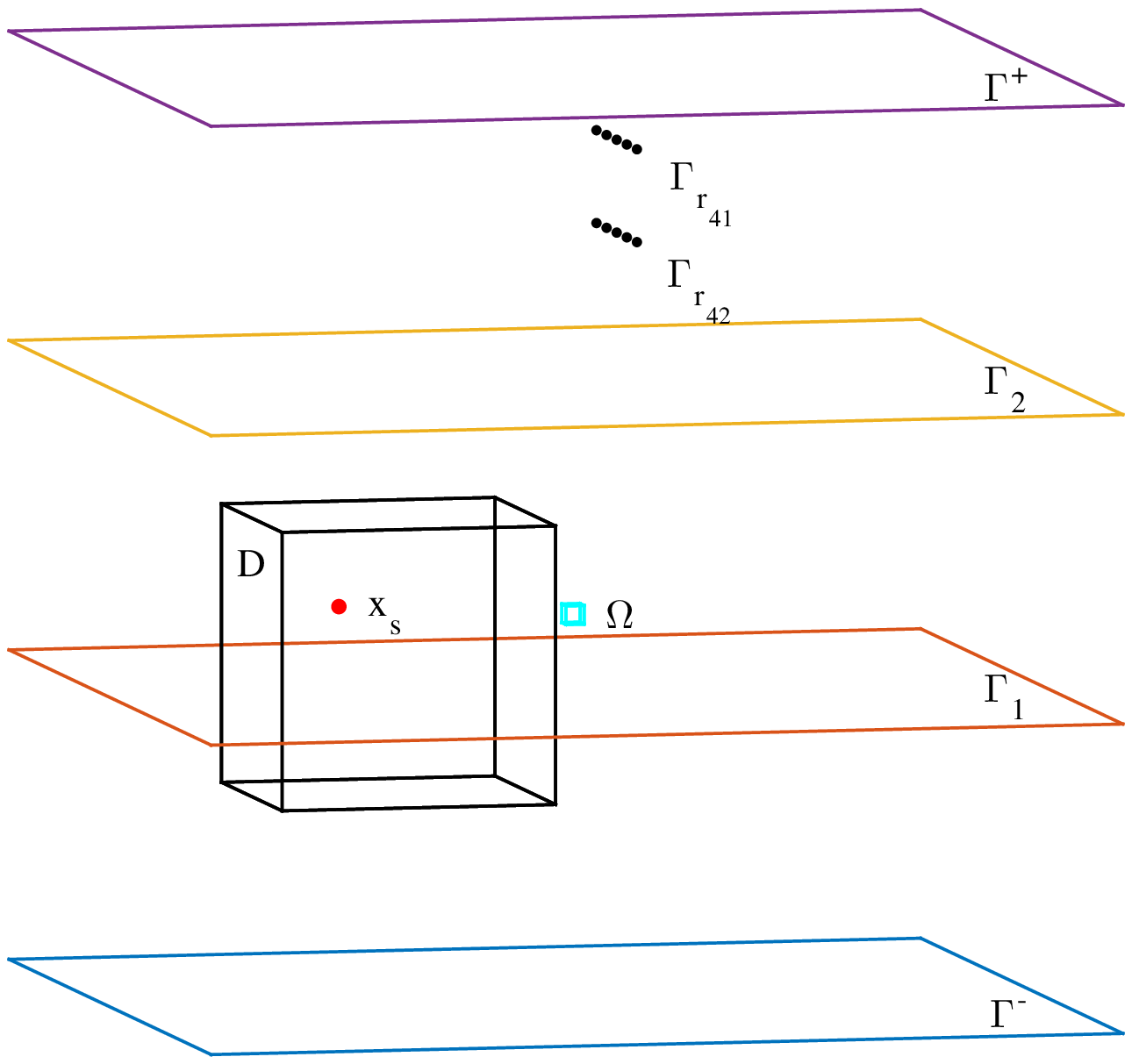}\hfill{}
 \vskip -0.6 truecm
 \caption{\label{fig:ex4_demo} \small{The location of the true point source $x_s$ in Example 4.}}
 \end{figurehere}

The reconstructions using the receivers on $\Gamma_{r_{41}}$ are shown in Figures \ref{fig:ex4}(a)-(c) 
while the ones using the receivers on $\Gamma_{r_{42}}$ are presented in Figures \ref{fig:ex4}(d)-(f). 
In Figures \ref{fig:ex4}(a)-(c), the position of the source is recovered, but some additional numerical 
artifact can be observed in the reconstruction. The artifact 
is due to the refraction of the scattered waves between layers $M_2$ and $M_3$. 
Then we place the detecting receivers in $\Gamma_{r_{42}}$ which is slightly closer to $\Gamma_2$, the location of the source is very accurately recovered.  
In spite of the refraction, the proposed method is able to provide a considerably reasonable estimation of the source's position with a very small number of receivers. 
Moreover, this example also explain that we may not obtain the accurate location of the source 
when the receivers are situated near the surface of the ocean.

\begin{figurehere}
 \hfill{}\includegraphics[clip,width=0.31\textwidth]{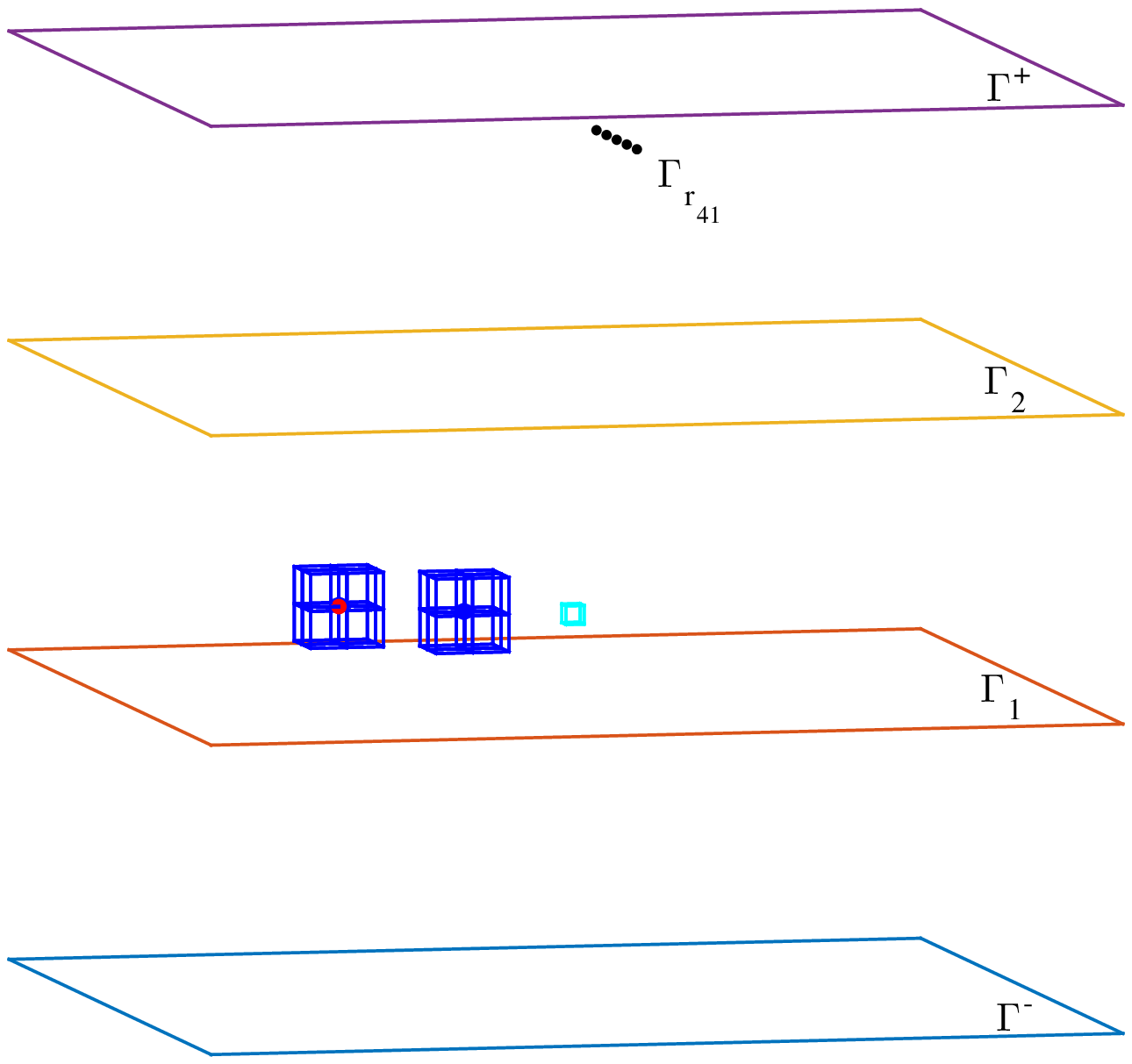}\hfill{}
 \hfill{}\includegraphics[clip,width=0.31\textwidth]{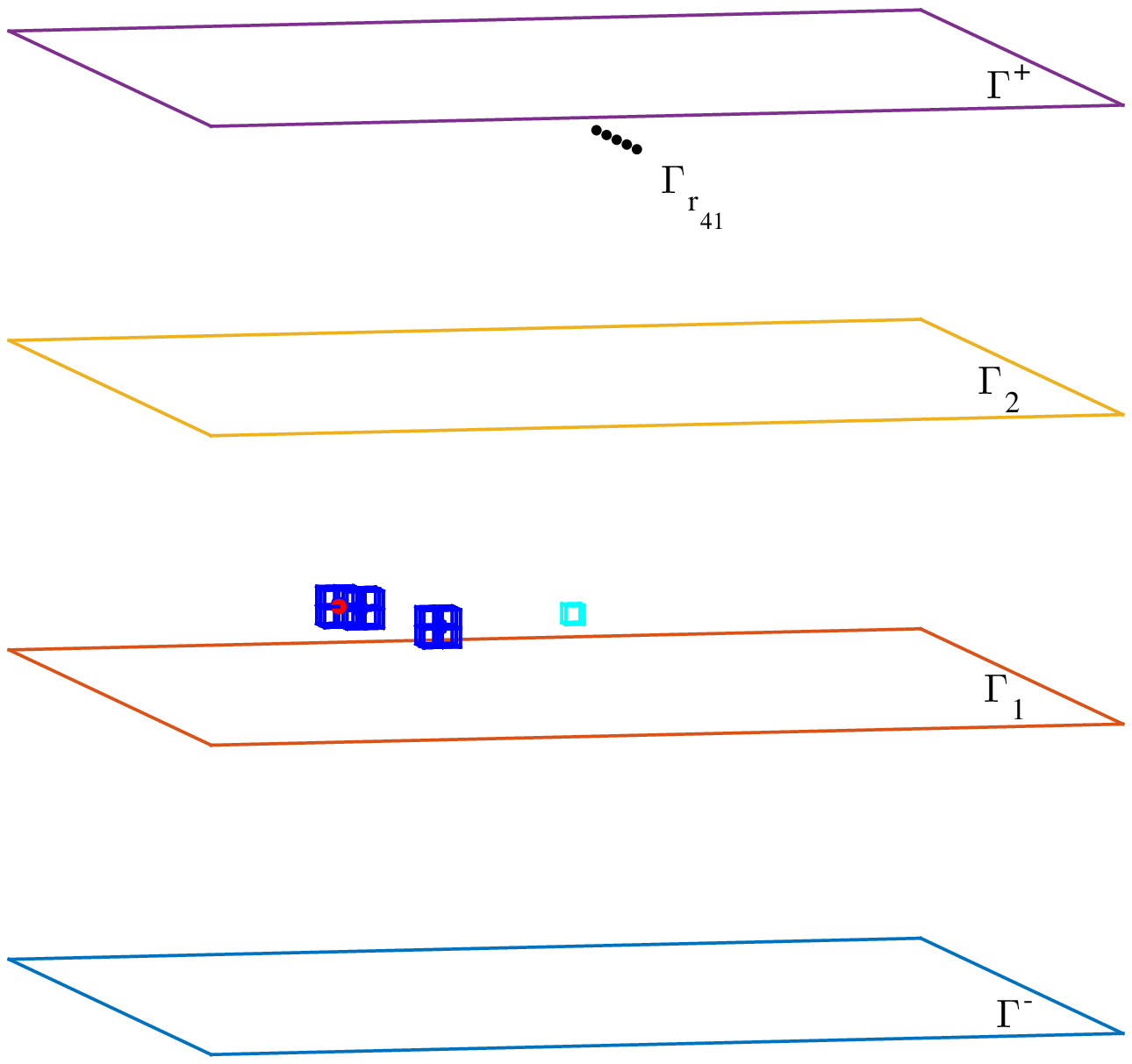}\hfill{}
 \hfill{}\includegraphics[clip,width=0.31\textwidth]{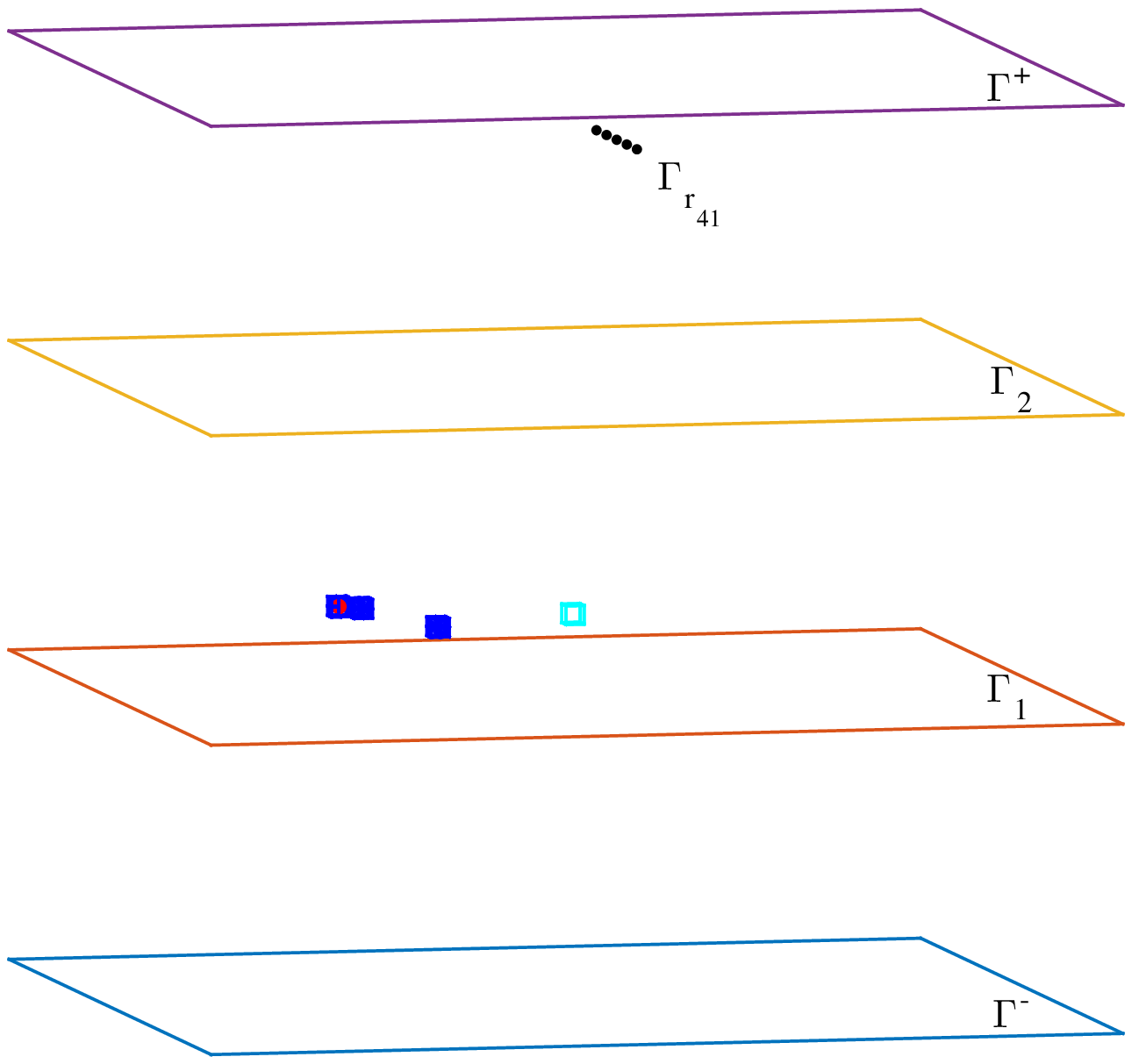}\hfill{}
 
 \hfill{}(a)\hfill{} \hfill{}(b)\hfill{} \hfill{}(c)\hfill{}
 
  \hfill{}\includegraphics[clip,width=0.31\textwidth]{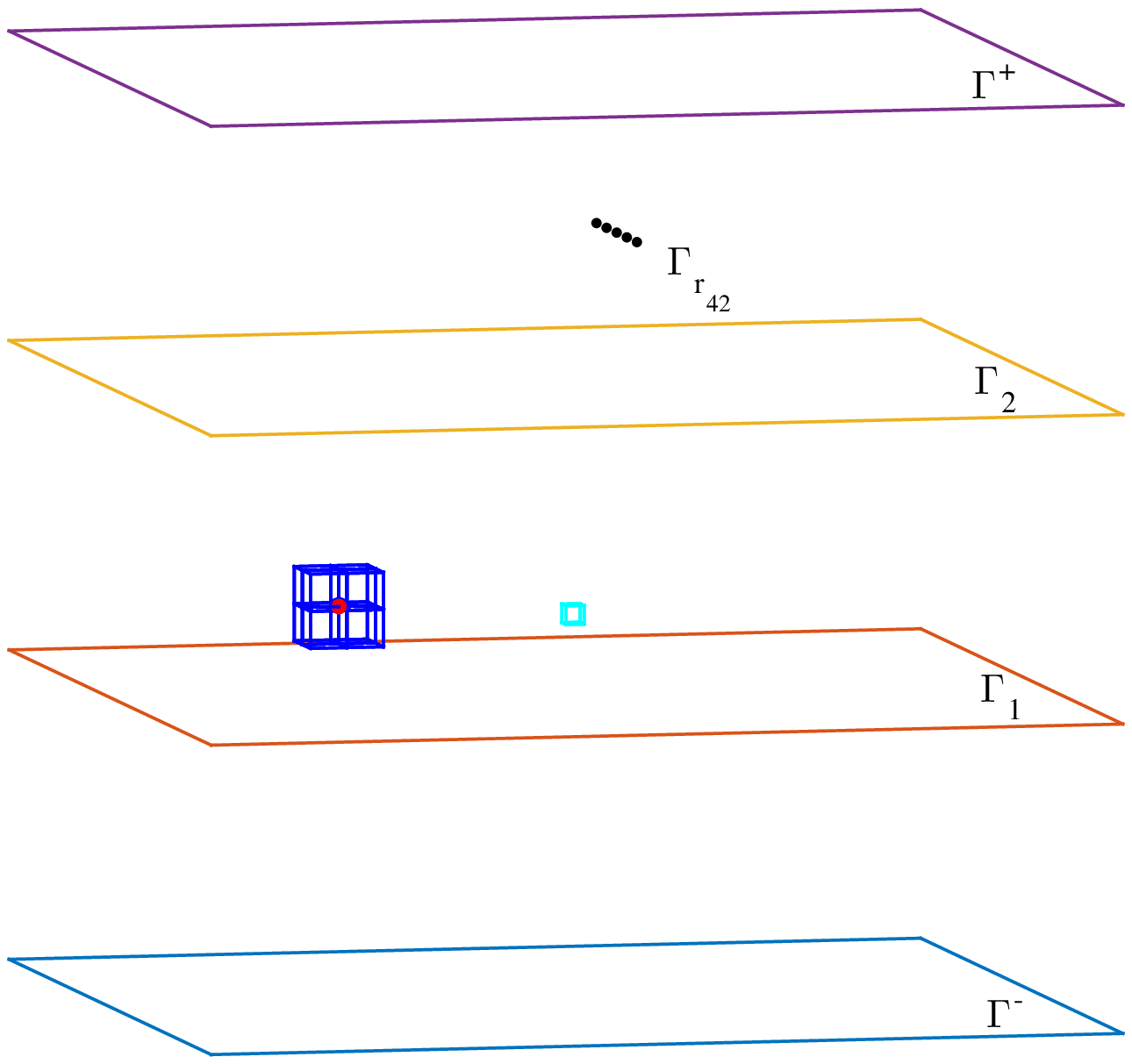}\hfill{}
 \hfill{}\includegraphics[clip,width=0.31\textwidth]{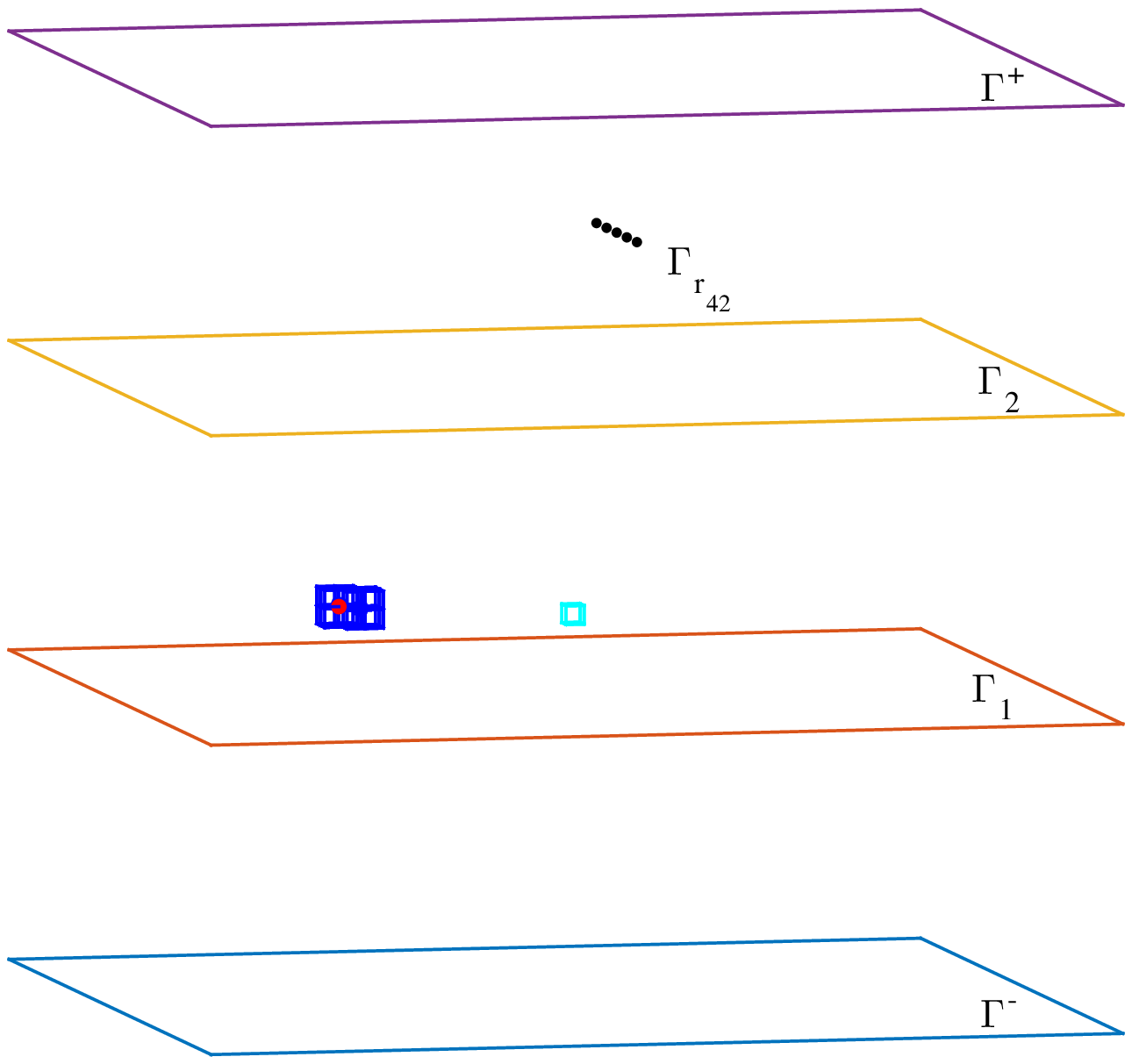}\hfill{}
 \hfill{}\includegraphics[clip,width=0.31\textwidth]{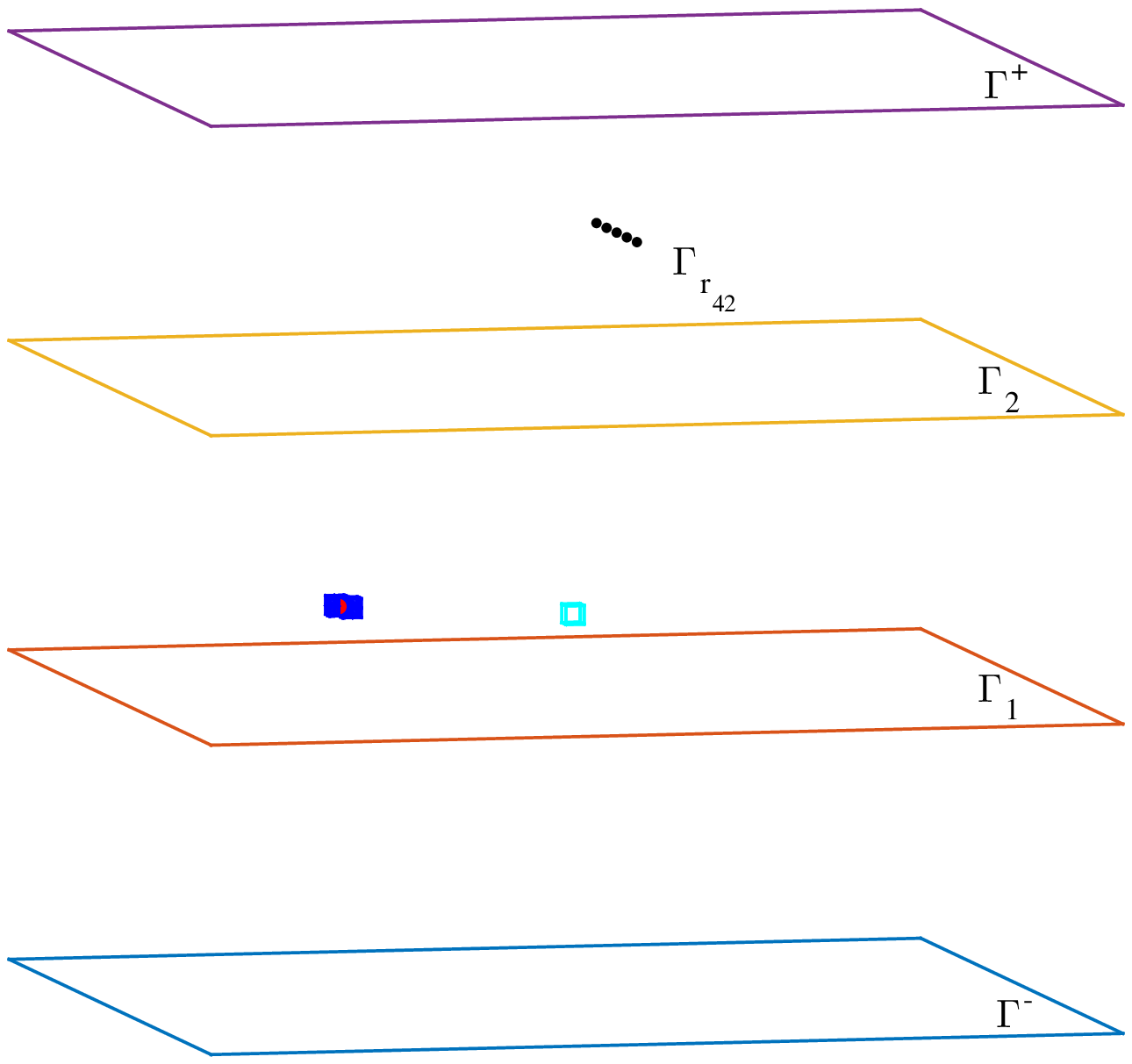}\hfill{}
 
 \hfill{}(d)\hfill{} \hfill{}(e)\hfill{} \hfill{}(f)\hfill{}
 
 \caption{\label{fig:ex4} \small{The reconstructions at the 1st to 3rd iterations in example 4 by $\Gamma_{r_{41}}$ (a)-(c) and by $\Gamma_{r_{42}}$ (d)-(f).}}
 \end{figurehere}

\section{Concluding remarks} \label{CR}
We have proposed a novel multilevel sampling algorithm 
to detect the position of a point source in a stratified ocean environment. The method exhibits 
several promising features: it is easy to implement, fast to converge and robust against noise. 
Moreover, the iterative algorithm can be viewed actually as a direct method, since it involves only matrix-vector operations 
and does not need any optimization process or the solution of any large-scale ill-posed linear systems. 
The method works with a very small number of receivers and its cut-off value is easy to select and can be fixed during the entire iterations. 
Moreover, this method may be modified for locating the sources in random acoustic
waveguide \cite{BIT}. 
Consequently, this method provides a simple and efficient alternative to detect the black box embedded in the stratified ocean. 


\end{document}